\pdfoutput=1

\documentclass[oneside]{amsart}
\usepackage[utf8]{inputenc}
\usepackage[T1]{fontenc}

\usepackage[dvipsnames]{xcolor}
\usepackage{todonotes}

\usepackage[all,arc,2cell]{xy}
\usepackage{tikz-cd}

\usepackage{hyperref}
\usepackage{geometry}
\usepackage{mathrsfs}
\usepackage{graphicx}

\usepackage{float}
\usepackage{color}

\usepackage{amssymb}
\usepackage{stmaryrd}

\usepackage{setspace}

\makeindex
\AtBeginDocument{}

\newcommand*{\Z}{{\mathbb Z}}
\newcommand*{\Q}{{\mathbb Q}}

\newcommand{\timesover}[1]{\underset{#1}{\times}}

\DeclareMathOperator{\Map}{Map}

\DeclareMathOperator{\Stab}{Stab}

\DeclareMathOperator{\Sub}{Sub}
\DeclareMathOperator{\Cl}{Cl}

\newcommand{\res}{res}
\newcommand{\tr}{tr}

\newcommand{\cotr}{\underline{cotr}}
\newcommand{\indtr}{\underline{\tr}}

\newcommand*{\m}[1]{{\protect\underline{#1}}}

\newcommand*{\mM}{\m{M}}

\newcommand*{\mA}{\m{A}}

\newcommand{\cc}[1]{\mathcal #1}

\newcommand{\cO}{\cc{O}}

\newcommand{\cU}{\cc{U}}

\newcommand{\id}{\textrm{id}}

\newcommand{\Ab}{\mathcal Ab}

\newcommand{\Mackey}{\mathcal Mackey}

\newcommand{\Coeff}{\mathcal Coe \mkern-2mu f \mkern-4mu f}

\newcommand{\Mod}{\mathcal Mod}
\newcommand{\dashMod}{\mhyphen\mathcal Mod}

\newcommand{\Fin}{\mathcal Fin}
\newcommand*{\cA}{\mathcal A}
\newcommand{\SubOG}{\Sub^{\cO}\!(G)}
\newcommand{\Orb}{\mathcal Orb}

\newcommand{\AboveHMackey}[1][{[H]}]{\mackabove_{#1}\mhyphen\Mackey}
\newcommand{\OMackey}{\cO\mhyphen\Mackey}

\mathchardef\mhyphen=45

\newcommand{\HOMackey}{\langle H \rangle^{\cO} \mhyphen \Mackey}

\numberwithin{equation}{section}

\newtheorem{theorem}{Theorem}[section]
\newtheorem{lemma}[theorem]{Lemma}
\newtheorem{corollary}[theorem]{Corollary}

\newtheorem{ex}[theorem]{Example}

\newtheorem{proposition}[theorem]{Proposition}

\newtheorem*{theorem*}{Theorem}
\newtheorem*{proposition*}{Proposition}
\newtheorem*{lemma*}{Lemma}

\newtheorem{MainTheorems}{Theorem}

\theoremstyle{remark}
\newtheorem{remark}[theorem]{Remark}
\newtheorem{notation}[theorem]{Notation}

\theoremstyle{definition}

\newtheorem{definition}[theorem]{Definition}

\DeclareMathOperator{\CoInd}{\textnormal{\textsf{CoInd}}}
\DeclareMathOperator{\Ind}{\textnormal{\textsf{Ind}}}
\DeclareMathOperator{\resfun}{\textnormal{\textsf{Res}}}

\newcommand*{\IndH}[1][{[H]}]{\Ind_{#1}}
\newcommand*{\CoIndH}[1][{[H]}]{\CoInd_{#1}}

\newcommand*{\IndHO}[1][{[H]^\cO}]{\Ind_{#1}}
\newcommand*{\IndNO}[1][{[N^{\cO}]}]{\Ind_{#1}}
\newcommand*{\CoIndHO}[1][{[H]^\cO}]{\CoInd_{#1}}

\newcommand*{\restrict}[2]{\resfun^{#1}_{#2}}

\DeclareMathOperator*{\bigdoublecurlyvee}{\bigcurlyvee\mkern-15mu\bigcurlyvee}
\newcommand{\mackabove}{{\bigdoublecurlyvee}}

\newcommand{\trivialgroup}{{\{1\}}}
\newcommand{\trivgp}{\trivialgroup}
\newcommand{\Mdef}[2]{\newcommand{#1}{\relax \ifmmode #2 \else $#2$\fi}}
\Mdef{\mcU}{\mathcal{U}}
\Mdef{\bQ}{\mathbb{Q}}
\Mdef{\mcL}{\mathcal{L}}

\newcommand{\blue}[1]{{\color{blue}#1}}

\definecolor{darkgreen}{HTML}{009B55}
\newcommand{\green}[1]{{\color{darkgreen}#1}}

\title{Splitting rational incomplete Mackey functors}
\author{David Barnes}
\address{Queen's University Belfast, UK}
\email{d.barnes@qub.ac.uk}

\author{Michael A. Hill}
\address{University of Minnesota, USA}
\email{mahill@umn.edu}

\author{Magdalena K\k{e}dziorek}
\address{Radboud University Nijmegen, Netherlands}
\email{m.kedziorek@math.ru.nl}

\begin{document}

\begin{abstract}
Inspired by equivariant homotopy theory, equivariant algebra studies     generalisations of \(G\)-Mackey functors that do not have all transfer maps     (also known as induction maps), for \(G\) a finite group. These incomplete Mackey functors have interesting and subtle properties that are more complicated than classical algebra. The levels of incompleteness that occur are indexed by simple combinatorial data known as transfer systems for \(G\), which are refinements of the subgroup relation satisfying certain axioms. The aim of this paper is to generalise the Greenlees--May and Th\'{e}venaz--Webb splitting result of rational \(G\)-Mackey functors to the incomplete case. 

By calculating idempotents of the rational incomplete Burnside ring of \(G\), we find the maximal splitting of the category of rational incomplete \(G\)-Mackey functors. These splittings are determined by maps of the form \(H \to G\) in the transfer system. We give an intrinsic definition of the split pieces beyond the idempotent description in order to understand what is the minimal information needed to determine an arbitrary rational incomplete \(G\)-Mackey functor. We end the paper with a series of examples of possible splittings and illustrate how simpler transfer systems have fewer terms in the splitting but the split pieces are more complicated. 
    
\end{abstract}

\maketitle

\tableofcontents

\section{Introduction}

\subsection{Motivation} Suppose \(G\) is a finite group. 
A \(G\)-Mackey functor is a collection of abelian groups \(M(G/H)\), 
one for each subgroup \(H\) of \(G\), together with transfer, restriction, and conjugation maps. This generalises the structure of the real (or complex) representation rings \(R(H)\).  Many standard constructions in representation theory 
can naturally be extended into a \(G\)-Mackey functor. 
For example, the Burnside ring \(A(G)\) of a group \(G\) and its subgroups \(A(H)\), for \(H\subseteq G\) form
a \(G\)-Mackey functor \(\mA\). Moreover, given any \(G\)-representation \(V\), the various 
\(H\)-fixed points (\(H\)-invariants) \(V^H\) for \(H\subseteq G\) assemble to 
give a \(G\)-Mackey functor. 

The study of \(G\)-Mackey functors began with Dress \cite{Dress73}
and Green \cite{Green71}. These ideas built on the notion of 
Frobenius functors from representation theory, as described in Curtis and Reiner \cite{CR94} and Bredon's coefficient systems, \cite{Bred67}, arising from equivariant stable homotopy theory. 
The homological algebra of \(G\)-Mackey functors  is a rich and complicated subject, see
Greenlees \cite{greenprojective92},  Mart\'inez-P\'erez and Nucinkis \cite{MPN06}
or Bouc, Stancu and Webb \cite{BSW17}. The situation is simpler in the rational case where we focus on \(G\)-Mackey functors that take values in rational vector spaces instead of abelian groups. 
It was shown independently by \cite{tw90} Th\'{e}venaz and Webb, and 
\cite{gremay95} Greenlees and May that the category of rational \(G\)-Mackey
functors for a finite group \(G\) has injective and projective dimension zero. Indeed, they prove that 
rationally the category of \(G\)-Mackey functors splits via an equivalence of categories
\[
\Mackey^G_{\mathbb{Q}} \cong \prod_{(H) \leq G} \mathbb{Q}[W_G H]\textrm{-modules}
\]
where the category on the right hand side is the product over conjugacy classes of subgroups of \(G\) of modules
over the rational group ring of the Weyl group of \(H\) in \(G\), \(W_G H = N_G(H)/H\). 
This is the categorification of the 
splitting of the rational Burnside ring of \(G\)
\[
A(G) \otimes \mathbb{Q} = \mA(G/G) \otimes \mathbb{Q} \cong \prod_{(H) \leq G} \mathbb{Q}.
\]

Inspired by equivariant homotopy theory, this paper studies incomplete \(G\)-Mackey functors: \(G\)-Mackey functors with fewer transfer maps. These are parameterized by transfer systems \(\cO\), and were defined by Blumberg--Hill. 
This is the additive version of bi-incomplete \(G\)-Tambara functors
of Blumberg and Hill \cite{BH-Bi-incompleteT} and they naturally arise as homotopy groups of rational incomplete \(G\)-spectra, as we discuss below.
The subject of this paper is a part of the new and growing area of equivariant algebra, 
where the study of modules and rings is translated to the study of 
(incomplete) \(G\)-Mackey functors, thought of as modules and
(incomplete) \(G\)-Tambara functors, thought of as commutative rings.

The aim of this paper is to extend the splitting result of rational 
\(G\)-Mackey functors to the incomplete setting, a vital step in
understanding the homological algebra of categories of incomplete rational \(G\)-Mackey functors.

\subsection{Main results}
We fix a finite group \(G\) and a transfer system \(\cO\) for \(G\). 
Let \(\OMackey^G\) denote the category of \(\cO\)-incomplete \(G\)-Mackey functors, see Definition \ref{def:OMackeyFunctors}. Informally, we may think of these as 
\(G\)-Mackey functors except one only has transfers for \(K \to H\) in \(\cO\). 

The \(\cO\)-Burnside ring of \(G\), \(\mA^{\cO}(G/G)\), is the Grothendieck ring of finite \(G\)-sets
\(G/H\) such that  \(H \to G\) is in \(\cO\). This is the endomorphism ring of the symmetric monoidal unit of \(\cO\)-Mackey functors, so any idempotent here gives a natural summand of the category.
In Section \ref{sec:idempotents} we study the idempotents of the rational 
\(\cO\)-Burnside ring of \(G\), \(\mA^{\cO}(G/G) \otimes \Q\), and their relation to the transfer system \(\cO\).

\begin{MainTheorems}
Let \(\SubOG\) be the set of subgroups \(H\) of \(G\) such that \(H \to G\) is in \(\cO\)
and let \(\SubOG/G\) be the set of conjugacy classes of \(\SubOG\).
For each \((H) \in \SubOG/G\) 
there is an idempotent \(e_{[H]^{\cO}} \in \mA^{\cO}(G/G) \otimes \Q\).  Together these 
give a maximal decomposition of the unit into orthogonal idempotents
\[
1 = \sum_{(H) \in \SubOG/G} e_{[H]^{\cO}} \in \mA^{\cO}(G/G) \otimes \Q.
\] 
\end{MainTheorems}

To unpack the terms of that sum somewhat, under the inclusion map 
\[
\mA^{\cO}(G/G) \otimes \Q \to \mA(G/G) \otimes \Q
\] 
\(e_{[H]^{\cO}}\) is sent to the sum of the usual idempotents \(e_{(K)}\)
for \(K\) in \([H]^{\cO}\): the set of those \(K\) which are \emph{inseparable} from \(H\) 
up to conjugacy (see Definition \ref{def:newinseparable}).
Formulas for the idempotents \(e_{[H]^{\cO}}\) appear as 
Lemma \ref{lem:idempotentformula} and Corollary \ref{cor:idempotentformula}.

\begin{MainTheorems}
The category of rational  \(\cO\)-Mackey functors for \(G\) splits into orthogonal full subcategories
\[ \OMackey^G_\bQ \cong \mA^\cO_\bQ \dashMod \simeq \prod_{(H)\in \SubOG/G}  e_{[H]^{\cO}}\mA^\cO_\bQ \dashMod. \]
\end{MainTheorems}
Moreover, this splitting is maximal, as can be seen by examining the rational \(\cO\)-Burnside Mackey functor \(\mA^\cO_\bQ\).

This theorem generalises the splitting method of Greenlees and May \cite[Appendix A]{gremay95}. In that reference, one can identify the split pieces in terms of modules over group rings. In our case the splitting is much more complicated and depends deeply on the structure of \(\cO\).

In Sections \ref{sec:transfersinsep} and \ref{sec:description_of_idempotent_pieces} we 
study the behaviour of inseparability classes and transfers, and construct several 
different subcategories of rational \(\cO\)-Mackey functors. By establishing adjunctions between these subcategories 
and proving an ambidexterity result for one such adjunction, we obtain an intrinsic understanding 
of the split pieces in Theorem \ref{thm:splitting2}. In Theorem \ref{thm: simplified-noconjugation} we simplify this further 
by removing duplicate information given by conjugacy. We combine and summarise those theorems to the following. 

\begin{MainTheorems}
The category \(e_{[H]^{\cO}}\mA^\cO_\bQ \dashMod \) is equivalent to the full subcategory of rational \(i^*_{N_G(H)}\cO\)-Mackey functors for \(N_G(H)\) restricted to those \(N_G(H)/K\) where \(K \subseteq H\) and which are zero on all \(K \notin [H]^{\cO}\). We call that category \(\HOMackey_\bQ\).
\end{MainTheorems}

The summand \(\HOMackey_\bQ \) depends on the transfers between elements of the inseparability class of \(H\). 
When there are no such transfers, the structure is much 
simpler as we see in Theorem \ref{thm:splitting2simplified_disklike}. 
Indeed, in Subsection \ref{subsec:reducedisklike} we study the case where 
\(\cO\) is generated by a set of transfers of the form \(H \to G\) for varying \(H\).
These are known as \emph{disk-like transfer systems} and for such a system there 
are no transfers in any inseparability class.
Corollary \ref{cor: simplified-[H]Mackey for disklike} gives the following result, where 
\(\Coeff_{\langle H\rangle^{\cO},\bQ}^N\) is the category of contravariant functors from 
the full subcategory of the orbit category for \(N=N_G(H)\) spanned by 
\(\{N/K \mid K\in [H]^{\cO}, K\subseteq H \}\) with values in \(\bQ\)-modules. 

\begin{MainTheorems} Let \(\cO\) be a disk-like transfer system. There is an equivalence of categories
     \[
        \OMackey_\bQ^G \cong \prod_{(H)\in \SubOG/G} \Coeff_{\langle H \rangle^{\cO},\bQ}^{N_G(H)}.
    \]
\end{MainTheorems}

The classical case of rational complete Mackey functors uses the transfer system \(\cO\) consisting of every subgroup inclusion. This transfer system is disk-like and the above splitting is precisely that of Greenlees--May. Indeed, the inseparability classes are exactly conjugacy classes and we recognise \(\Coeff_{\langle H \rangle^{\cO},\bQ}^{N_G(H)}\) as the category of rational modules with an action of \(W_G H\), the Weyl group of \(H\) in \(G\).

We conclude the paper with a series of examples for \(C_6\) and \(C_8\) that illustrate our results. We are able to study these splittings in more detail, as in the abelian case the results simplify substantially. Our examples show that the simpler the transfer system, the fewer terms there are in the splitting and the more complicated the structure of the split pieces. 

While we work rationally for simplicity, one only needs to work over a ring where the order of \(G\), \(|G|\), is a unit.  For integral calculations, one could combine our results with the method of Liu \cite{Liu23}.  A further generalisation would be to consider modules over general Green functors, see Bouc, Dell’Ambrogio and Martos \cite{GeneralGM2024} for splittings in the complete case.

\subsection{Connection to equivariant stable homotopy theory}

The categories of \(G\)-Mackey functors and \(G\)-spectra (\(G\) a finite group) are well-known to be deeply related. 
For example, the equivariant homotopy groups of a genuine \(G\)-spectrum \(X\) determine a \(\mathbb{Z}\)-graded \(G\)-Mackey functor:
\(G/H \mapsto \pi_*^H(X)\). 
Conversely, given any \(G\)-Mackey functor \(\mM\), one can make an equivariant Eilenberg-Mac Lane spectrum \(H\mM\), 
which has homotopy groups concentrated in degree zero and satisfies \(\pi_0^H(H\mM) = \mM(G/H)\).
When one works with \emph{rational} \(G\)-Mackey functors and \emph{rational} \(G\)-spectra, this relation becomes 
even stronger, in that rational \(G\)-Mackey functors provide an ``algebraic model'' for rational \(G\)-spectra, 
see \cite[Appendix A]{gremay95},  work of the first author \cite{barnesfinite},
and work of the third author \cite{KedziorekExceptional}. 

Recent work of Blumberg and the second author \cite{BH_Equiv_homot_theory, BH-Bi-incompleteT}
has focused on \(G\)-spectra that are not genuine, but also are not na\"ive (here by \emph{na\"ive} we mean spectra with a \(G\)-action). 
Instead, one takes an \(N_\infty\)-operad 
and uses this to define the \emph{additive} structure on \(G\)-spectra. 
The resulting homotopy category of \emph{incomplete} \(G\)-spectra has transfer maps and Wirthm\"uller isomorphisms 
according to the transfer system \(\cO\) associated to the chosen \(N_\infty\)-operad, see \cite{BH_Norms, Rubin, BonventrePereira,  GutierrezWhite}. 
In particular, the homotopy groups of such an incomplete \(G\)-spectrum naturally form an \(\cO\)-Mackey functor.
Hence, calculations in the \(\cO\)-incomplete equivariant stable homotopy category,
such as spectral sequence calculations, will need an understanding of the category of \(\cO\)-Mackey functors. 
In particular, the homotopy groups of the sphere spectrum are given by the \(\cO\)-Burnside Mackey functor of \(G\), \(\mA^\cO\).

These categories of incomplete \(G\)-spectra are less exotic than they may seem at first. 
In the case of the equivariant little disks \(N_\infty\)-operad 
for \(\cU\) an incomplete \(G\)-universe, 
the corresponding homotopy category of \(G\)-spectra is modelled by the somewhat familiar 
category of orthogonal \(G\)-spectra indexed on \(\cU\).
These homotopy categories for incomplete universes appear in work of Lewis \cite{Lewis00splitting} and 
work of Yavuz \cite{Yavuz25} on \(C_2\)-equivariant orthogonal calculus.
Another example of work that uses \(G\)-spectra indexed on incomplete universe is Manolescu's Pin(2)-equivariant approach to the \(11/8\)-conjecture \cite{Manolescu}. When restricted to \(C_4\subset S^1\subset \text{Pin(2)}\) this gives \(C_4\)-spectra indexed on an incomplete (but not trivial) universe.

Finally we note that the splitting results and intrinsic description of the split pieces provide a model for
the ongoing work by Bohmann and the authors on splitting the category of rational \(\cO\)-incomplete spectra and studying the resulting components. 
The lifting of our algebra-first approach to the topological setting is made concrete by  
the description of incomplete \(G\)-spectra as incomplete spectral Mackey functors of Smith \cite{Smith_thesis}.

\subsection{Notation}
Let \(\Sub(G)\)\index{sub@\(\Sub(G)\)} denote the poset of subgroups of \(G\). Let \(\Fin^G\)\index{finite@\(\Fin^G\)} denote the category of finite \(G\)-sets and \(G\)-equivariant maps, and let \(\Orb^G\)\index{@orb\(\Orb^G\)} denote the orbit category of \(G\). 

Throughout the paper \(\cO\)\index{operad@\(\cO\)} denotes a transfer system for \(G\), see Definition \ref{def:transfersystem}.
At the end we include Index of notation used in the paper.

\subsection{Acknowledgements}
The authors are grateful to Hausdorff Research Institute for Mathematics in Bonn funded by the Deutsche Forschungsgemeinschaft (DFG, German Research Foundation) under Germany's Excellence Strategy - EXC-2047/1 - 390685813 and to the organisers of the trimester program \emph{Spectral Methods in Algebra, Geometry, and Topology} 2022. 
The second author was supported  by the National Science Foundation under Grant No. 2105019 \& 2528364.
The third author was supported by the Nederlandse Organisatie voor Wetenschappelijk Onderzoek (Dutch Research Council) Vidi grant no VI.Vidi.203.004. The authors thank the anonymous referees for their useful suggestions.

\section{Background}
\subsection{Transfer systems and admissible sets}
We begin this section by recalling the definition of a transfer system after \cite{RubinDetecting}, \cite{BarnesBalchinRoitzheim}.

\begin{definition}\label{def:transfersystem}
    Let \(G\) be a finite group. A \emph{\(G\)-transfer system} \(\cO\)\index{operad@\(\cO\)} is a partial order on \(\Sub(G)\) denoted by arrows \(\to\), which refines the subset relation, and which is closed under 
    \begin{itemize}
        \item conjugation: if \(K\to H\), then 
        \((gKg^{-1}) \to (gHg^{-1})\) for every
group element \(g \in G\), and
        \item restriction: if \(K\to H\) and \(L \subseteq H\), then \((K \cap L) \to L\).
    \end{itemize}
 If  \(K\to H\), then we say that \(H/K\) is \(\cO\)-\emph{admissible}.
\end{definition}

\begin{remark}
We thank the anonymous referees for pointing out that the definition of a \emph{Mackey system} by Bley and Boltje,  \cite[Definition 2.1]{BB04}, is equivalent to the definition of a transfer system when \(G\) is a finite group. This much earlier reference was unknown to the authors.

Mackey systems were defined for any group \(G\) and Bley and Boltje use them to construct generalisations of Mackey functors. While their formalism does allow for the incomplete Mackey functors we study in this paper, Bley and Boltje mainly focus on the case of infinite groups.    
\end{remark}

The restriction axiom of Definition \ref{def:transfersystem} shows that when we look at the slice categories of all things which transfer to a fixed element, then these are closed under meets.

\begin{lemma}[Meet Lemma, {\cite[Lemma A.6]{RubinDetecting}}]\label{lem: Meet Lemma} If \(H \to K\) and \(J \to K\) are both in the transfer system \(\mathcal{O}\), then \(H\cap J\to H\) and \(H\cap J \to J\) are in \(\mathcal{O}\).  
\end{lemma} 

Note that transfer systems for \(G\) form a poset with respect to inclusion.

Given a transfer system \(\cO\) for \(G\) and a subgroup \(H\) of \(G\), there is an associated transfer system on \(H\) formed by restriction.

\begin{definition}
    If \(\cO\) is a transfer system for \(G\), then let \(i_H^\ast\cO\) be the restriction of \(\cO\) to \(\Sub(H)\).
\end{definition}

\begin{remark}
    For any subgroup \(H\), the restriction \(i_H^\ast\cO\) not only has conjugation invariance by \(H\) but also a stronger condition: if \(K\subseteq H\) and \(g\in G\) is such that \(gKg^{-1}\subseteq H\), then whenever \(J\to K\), we also have \(gJg^{-1}\to gKg^{-1}\). In particular, we have conjugation invariance for conjugation by any element in the normalizer of \(H\) in \(G\).
\end{remark}

\subsection{\texorpdfstring{\(\cO\)-Mackey functors}{O-Mackey functors}}
Associated to any transfer system is a naturally defined version of Mackey functors where the transfers are essentially parameterized by the maps in the transfer system. For this, we build a version of the Lindner category, starting with a subcategory of \(\Fin^G\) that will structure transfers. In the language of \cite{BH-Bi-incompleteT}, 
\(\Fin^{G}_{\cO}\) is the indexing category
associated to \(\cO\).
\begin{definition}
    Let \(\Fin^{G}_{\cO}\)\index{finite@\(\Fin^{G}_{\cO}\)} be the wide subcategory of the category \(\Fin^G\) of finite \(G\)-sets where a map \(f\colon S\to T\) of finite \(G\)-sets is in \(\Fin^G_{\cO}\) if and only if for all \(s\in S\), we have
    \[
        \Stab(s)\to \Stab\!\big(f(s)\big)
    \]
    in \(\cO\).
\end{definition}

Put another way, this is the smallest wide, pullback stable, finite coproduct complete subcategory of the category of finite \(G\)-sets that contains all morphisms \(G/K \to G/H\) whenever \(K \to H\) is in \(\cO\).

\begin{definition}[{\cite[Definition 7.21]{BH-Bi-incompleteT}}]
Define \(\cA^{\cO}\)\index{A@\(\cA^{\cO}\)}, the Lindner category of \(\cO\),  to have objects the finite \(G\)-sets and morphisms the isomorphism classes of spans 
\[
\cA^{\cO}(S,T) = \big\{ [S \leftarrow U \xrightarrow{h} T] \mid h \in \Fin^{G}_{\cO}\big\}.
\]
\end{definition}

The category \(\cA^{\cO}\) is pre-additive, and the categorical biproduct is given by the disjoint union of finite \(G\)-sets.
When we take the complete transfer system consisting of all subgroup inclusions, 
\(\cA^{\cO}\) is the usual Lindner category of spans of \(G\)-sets.

\begin{definition}[{\cite[Definitions 7.21 and 7.24]{BH-Bi-incompleteT}}]\label{def:OMackeyFunctors}
    Let \(G\) be a finite group and \(\cO\) be a transfer system for \(G\). 
    An \(\cO\)-Mackey functor is a product preserving functor 
    \[M: \cA^{\cO} \to \Ab.\] 
    A map of \(\cO\)-Mackey functors is a natural transformation. We will use the notation \(\OMackey^G\)\index{omackey@\(\OMackey^G\)} for this category.
\end{definition}

Notice that this definition generalises two classical cases; when \(\cO\) is the maximal transfer system we recover the category of complete \(G\)-Mackey functors and when \(\cO\) is the minimal transfer system we recover the category of \(G\)-coefficient systems. 
As Mackey functors are  product preserving functors, we only need to consider the values of our Mackey functors on orbits \(G/H\).

A key example is the incomplete version of the Burnside Mackey functor.
\begin{ex}[{\cite[Definition 7.25]{BH-Bi-incompleteT}}]\label{ex:incompleteburnside}
   The \(\cO\)-Burnside Mackey functor \(\mA^{\cO}\)\index{A@\(\mA^{\cO}\)} takes value 
   \[
    \mA^{\cO}(G/H) = \Z \big\{ H/K \mid K \to H \textrm{ in } \cO   \big\},
   \]   
   the free abelian group on isomorphism classes of \(\cO\)-admissible orbits \(H/K\).
    The conjugation and restriction maps are as for the complete Burnside Mackey functor.  
    We have a transfer map 
    \[
    \mA^{\cO}(G/K) \to \mA^{\cO}(G/H)
    \]
    if and only if
    \(K \to H\) is in \(\cO\). This transfer map is given by the induction
    construction on \(K\)-sets, sending \(K/L\) to \(H/L\). 
\end{ex}

\begin{proposition}
    The category of \(\cO\)-Mackey functors is symmetric monoidal under the Day convolution with unit the \(\cO\)-Burnside Mackey functor, hence we have an equivalence of categories
    \[
    \OMackey \cong \mA^\cO \dashMod.
    \]
\end{proposition}
    This equivalence passes to the rational versions
    \[
    \OMackey_\bQ \cong \mA^\cO_\bQ \dashMod,
    \]
    where \(\OMackey_\bQ\) is the category of \(\cO\)-Mackey functors with values in rational vector spaces     
    and \(\mA^\cO_\bQ\) is the rational \(\cO\)-Burnside ring Mackey functor, that is, 
    \(\mA^\cO_\bQ(G/H) = \mA^\cO(G/H) \otimes \bQ\).

\begin{ex}\label{ex:c_6_Burnside}
    Let \(G=C_6\) and take the transfer system consisting of the identity maps, \(C_2 \to C_6\) and \(C_1 \to C_3\). 
    The rational \(C_6\)-Burnside ring for \(\cO\) is the following incomplete \(\cO\)-Mackey functor, where in the square brackets we indicate the additive basis of each \(\Q\)-module. We draw this as a Lewis diagram below.

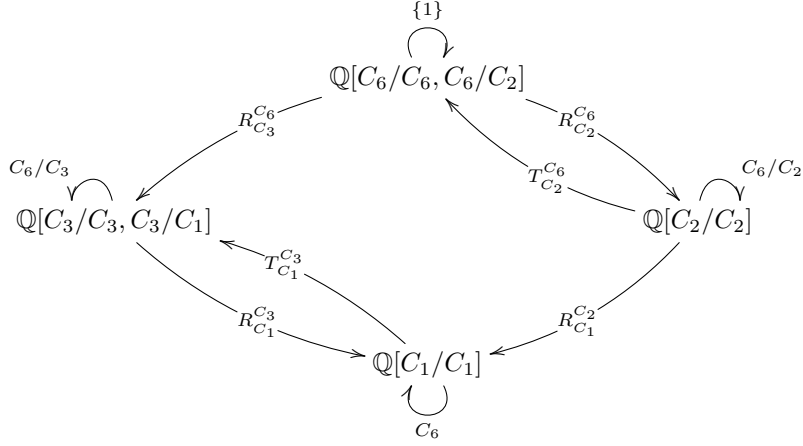
\[
\xymatrix@R+0.5cm@C+0.5cm{
& \bQ[C_6/C_6, C_6/C_2]
\ar@(ul,ur)^{\trivgp}
\ar@/^1.pc/[dr]|{R_{C_2}^{C_6}}
\ar@/_1.pc/[dl]|{R_{C_3}^{C_6}}
\\
\bQ[C_3/C_3,C_3/C_1]
\ar@(u,l)_{C_6/C_3}
\ar@/_1.pc/[dr]|{R_{C_1}^{C_3}}
&&
\bQ[C_2/C_2]
\ar@(u,r)^{C_6/C_2}
\ar@/^1.pc/[dl]|{R_{C_1}^{C_2}}
\ar@/^1.pc/[ul]|{T_{C_2}^{C_6}}
\\
&
\bQ[C_1/C_1]
\ar@(dr,dl)^{C_6}
\ar@/_1.pc/[ul]|{T_{C_1}^{C_3}}
}
\]
\end{ex}

\section{Splitting the category of rational \texorpdfstring{\(\cO\)}{O}-Mackey functors}\label{sec:idempotents}

\subsection{The Rational \texorpdfstring{\(\cO\)}{O}-Burnside ring}

\begin{definition}\label{def:Sub_OG}
	If \(\cO\) is a transfer system, let \(\SubOG\)\index{sub@\(\SubOG\)} be the subposet of \(\Sub(G)\) consisting of subgroups \(H\) such that \(H\to G\) in \(\cO\). 
\end{definition}

\begin{proposition}
	The inclusion of \(\SubOG\) into \(\Sub(G)\) preserves meets. 
\end{proposition}
\begin{proof}
    This is a restatement of Lemma~\ref{lem: Meet Lemma}.
\end{proof}

\begin{corollary}\label{cor:minimal_transferingtoG}
    There is a minimal element \(N^{\cO}\) in \(\SubOG\).
\end{corollary}

The conjugation condition on a transfer system also implies that \(\Sub^G(\cO)\) is closed under conjugation.

\begin{proposition}
    The sub-poset \(\SubOG\) is a \(G\)-equivariant sub-poset of \(\Sub(G)\) under conjugation.
\end{proposition}

\begin{definition}
    An \(\cO\)-class function is a \(G\)-equivariant function
    \[
        \SubOG\to \mathbb Z.
    \]
    This is a ring under point-wise sums and products, and we denote this ring by \(\Cl^{\cO}(G;\Z)\)\index{class@\(\Cl^{\cO}(G;\Z)\)}.
\end{definition}

\begin{proposition}\label{prop: markhom}
	 If \(\cO\) is a transfer system, the ghost coordinate map 
	\[
		\mA^{\cO}(G/G)\to \Cl^{\cO}(G;\Z)
	\]\index{A@\(\mA^{\cO}\)}
    given by 
    \[
        T\mapsto \big(H\mapsto |T^H|\big)
    \]
	is an injective map of rings and an isomorphism after rationalization.
\end{proposition}
\begin{proof}
The proof is the same as in the classical, complete case. The usual argument that the map can be represented as an upper triangular matrix after choosing an ordering on \(\SubOG\) compatible with the cardinalities goes through without change, which shows that this map is injective. After rationalization, we get vector spaces of the same finite dimension on both sides, thus the injective map is an isomorphism.
\end{proof}

\begin{corollary}\label{cor: Separates Points}
    Given two subgroups \(H_0,H_1\in \SubOG\), the subgroups are conjugate if and only if the functions on \(\SubOG\) defined by
    \[ 
    K \mapsto |(G/K)^{H_0}|\text{ and }K\mapsto |(G/K)^{H_1}|
    \]
    agree.
\end{corollary}

As a consequence, we immediately obtain a description of minimal idempotents in \(\mA^{\cO}_\Q(G/G)\) and hence a maximal decomposition of the unit.

\begin{proposition}
    The minimal idempotents in \(\mA^{\cO}_\Q(G/G)\) correspond to the characteristic functions in \(\Cl^{\cO}(G;\Z)\)
    \[
    \delta_{H}(K)=\begin{cases} 1 & K=gHg^{-1}, \\
    0 & \text{otherwise.}
    \end{cases}
    \] 
\end{proposition}

\begin{definition}
    Let \(e_{[H]^\cO}\)\index{e@\(e_{[H]^{\cO}}\)} denote the idempotent in \(\mA^{\cO}_{\bQ}(G/G)\) corresponding to \(\delta_H\). If the transfer system \(\cO\) is clear from the context, we may remove it from the notation of the idempotent.
\end{definition}

We adapt Gluck's approach to the incomplete case and obtain formulas relating the idempotents to the admissible orbits, see \cite[Section 3]{gluck}.

\begin{lemma}\label{lem:orbitsumidempotent}
    For \(H \to G\) in \(\cO\), we have a decomposition in \(\mA^{\cO}_{\Q}(G/G)\):
    \[
    G/H  = 
    \sum_{K \to H} 
    \frac{|N_G(K)|}{|H|} e_{[K]^{\cO}}. 
    \]
\end{lemma}
\begin{proof}
    This follows by noting that the ghost coordinate map applied to \(G/H\) 
    evaluated at \(K\) (a subgroup of \(H\) with \(K \to G \in \cO\)) is \(| (G/H)^K |\). 
    We may calculate the size of that set as the number of those \(gH \in G/H\) where 
    \(g^{-1}Kg \subseteq H\). Let
    \[
    S = \{ g \in G \mid g^{-1}Kg \subseteq H \}
    \]
    then \(| (G/H)^K | = |S|/|H| \in \Z\).
    The set \(S\) may be partitioned into cosets of \(N_G(K)\), one coset for each
    distinct \(G\)-conjugate of \(K\) that is inside \(H\).
    Since \(e_{[K]^{\cO}}\) is invariant under conjugation of \(G\), the formula follows. 
\end{proof}

Since \(\{K\mid K\to H \in \cO\}\) is a poset under \(\cO\), M\"obius inversion gives us a way to use this to write the idempotents \(e_{[H]^{\cO}}\) in terms of the natural permutation basis of \(\mA^{\cO}_{\Q}(G/G)\). For this, recall the M\"obius function in the poset.
\begin{definition}
    If \(K\to H\) is in \(\cO\), then let
    \[
        \mu^{\cO}(K,H)=\sum_{i} (-1)^i c_i,
    \]\index{mobius@\(\mu^{\cO}(K,H)\)}
    where \(c_i\) is the number of strictly increasing chains of subgroups in \(\cO\)
    \[
    K=H_0\to H_1\to\dots\to H_{i-1} \to H_i = H, 
    \]
    and where \(\mu^{\cO}(K,K)=1\).
\end{definition}

\begin{lemma}\label{lem:idempotentformula}
    For  \(H \to G\) in \(\cO\), we have a decomposition
    \[
        e_{[H]^{\cO}} = \sum_{K \to H} \frac{|K|}{|N_G H|} \mu^\cO(K,H) G/K.
    \]
\end{lemma}
\begin{proof}
    Define \(f_{[H]} = |N_G(H)|e_{[H]^{\cO}}\) and 
    \(v_{[H]} = |H|G/H\). Then Lemma \ref{lem:orbitsumidempotent} may be written as 
    \[
    v_{[H]}  = \sum_{K\to H} 
    f_{[K]}.
    \]
    The M\"obius inversion formula gives 
    \[
    f_{[H]} = \sum_{K\to H} 
    \mu^\cO(K,H) v_{[K]}. \qedhere
    \]    
\end{proof}

The transfer maps of the incomplete Burnside ring are given by induction 
\[
    H/K \mapsto G \times_H H/K = G/K.
\]
Since all summands in the formula of Lemma \ref{lem:idempotentformula} are of the form \(G/K\) with \(K\) a subgroup of \(H\), we obtain the following.
\begin{corollary}\label{cor:idempotenttransfer}
    For \(H \to G\) in \(\cO\), the idempotent \(e_{[H]^{\cO}}\) is in the image of the 
    transfer map \(\mA^{\cO}_{\Q}(G/H) \to \mA^{\cO}_{\Q}(G/G)\):
    \[
        e_{[H]^{\cO}}=\tr_{H}^{G}\left(\sum_{K\to H}\frac{|K|}{|N_GH|}\mu^{\cO}(K,H) H/K\right).
    \]
\end{corollary}

An idempotent \(e\) of \(\mA^{\cO}_{\Q}\) splits the rational \(\cO\)-Burnside ring Mackey functor
\[
\mA^\cO_\bQ \cong e\mA^\cO_\bQ \oplus (1-e)\mA^\cO_\bQ
\]
where \((e\mA^\cO_\bQ)(G/K) = \res_K^G (e) \mA^\cO_\bQ(G/K) \).

\begin{corollary}\label{cor:splitting O Mackey}
    For \(\cO\) a transfer system we have a splitting of the category of rational \(\cO\)-Mackey functors
\[  \mA^\cO_\bQ \dashMod \simeq \prod_{(H)\in \SubOG/G}  e_{[H]^{\cO}}\mA^\cO_\bQ \dashMod,
\]
 where the product is taken over conjugacy classes of \(H\), with \(H \to G\) in \(\cO\).

    Thus every rational \(\cO\)-Mackey functor splits 
    \[
        \mM\cong \bigoplus_{(H)\in \SubOG/G} e_{[H]^{\cO}}\mM.
    \]
\end{corollary}

To understand the splitting, we need to understand how these interact with groups outside of \(\SubOG\). For this, we compare with the complete case, the case of the terminal transfer system.

\begin{proposition}
    The inclusion \(\SubOG\to \Sub(G)\) induces a surjective map of rings
    \[
        \Cl(G;\Z)\to \Cl^{\cO}(G;\Z).
    \]
\end{proposition}
\begin{proof}
    Since \(\SubOG\) is a \(G\)-equivariant subset of \(\Sub(G)\), this is clear. Surjectivity follows from the observation that every sub-\(G\)-set of a \(G\)-set is a summand.
\end{proof}

In fact, we have a preferred section for this projection map, see Corollary \ref{cor: Class Functions Split}. This will be crucial in understanding how subgroups of \(G\) are grouped together.

\begin{proposition}\label{prop: Minimal Admissible for J}
    Given a subgroup \(J\), there is a unique minimal \(H\supseteq J\) such that
    \begin{enumerate}
        \item \(H\in\SubOG\) and
        \item if \(K\in\SubOG\) and \(J\subseteq K\), then \(H\subseteq K\).
    \end{enumerate}
\end{proposition}
\begin{proof}
    Consider 
    \[
    \mathcal J=\{H\mid J\subseteq H, H\to G\}.
    \]
    This is a finite, non-empty set (since \(G\) is always in \(\mathcal J\)), all elements of which transfer up to \(G\). This means we have
    \[
        \bigcap_{H\in\mathcal J}H\to G,
    \]
    by the Meet Lemma and \(J\) is visibly in the intersection.
\end{proof}

\begin{notation}
    Let \(H_{\cO}(J)\)\index{Hoperad@\(H_{\cO}(-)\)} be this unique minimal \(H\) for \(J\).
\end{notation}

Minimality actually implies a much stronger condition connecting the normalizers of various subgroups of \(G\).

\begin{proposition}\label{prop: J not in any distinct conjugate of H}
    Let \(J\) be a subgroup of \(G\).
    If \(g\not\in N_G(H_{\cO}(J))\), then \(J\not\subseteq gH_{\cO}(J)g^{-1}\). 
\end{proposition}
\begin{proof}
    Both \(H_{\cO}(J)\) and \(gH_{\cO}(J)g^{-1}\) transfer up to \(G\), and so \(J\) would be in the intersection, which is strictly smaller than \(H_{\cO}(J)\) and transfers up to \(G\) by the Meet Lemma.
\end{proof}

\begin{proposition}\label{prop: Retraction is monotone}
    If \(J\subseteq K\), then \(H_{\cO}(J)\subseteq H_{\cO}(K)\).
\end{proposition}
\begin{proof}
    Since \(J\subseteq K\), we have an inclusion
    \[
        \big\{H\mid H\in\SubOG, K\subseteq H\big\}\subseteq
        \big\{H\mid H\in\SubOG, J\subseteq H\big\}.
    \]
    This gives the desired inclusion.
\end{proof}

\begin{corollary}
    The assignment
    \[
        J\mapsto H_{\cO}(J)
    \]
    gives a retraction of \(\Sub(G)\) onto \(\SubOG\).
\end{corollary}

\begin{corollary}\label{cor: Class Functions Split}
    The projection
    \[
        \Cl(G;\Z)\to \Cl^{\cO}(G;\Z)
    \]\index{class@\(\Cl^{\cO}(G;\Z)\)}
    induced by the inclusion of \(\SubOG\) into \(\Sub(G)\) is split by precomposition with \(H_{\cO}\):
    \[
        H_{\cO}^{\ast}\colon \Cl^{\cO}(G;\Z)\hookrightarrow \Cl(G;\Z).
    \]
\end{corollary}

\begin{definition}\label{def:newinseparable}
    We say that two subgroups \(J\) and \(K\) of \(G\) are {\emph{inseparable}} if we have an equality of conjugacy classes
    \[
        (H_{\cO}(J))=(H_{\cO}(K))\in\SubOG/G.
    \]\index{Hoperad@\(H_{\cO}(-)\)}
\end{definition}

Since these are defined as the things identified by some function, inseparability is an equivalence relation. 

\begin{proposition}
    The relation \(H\sim_{\cO} K\) defined by ``\(H\) and \(K\) are inseparable for \(\cO\)'' is an equivalence relation. We denote by \([H]^{\cO}\)\index{H@\([H]^\cO\)} the equivalence class of \(H\) with respect to this relation.
\end{proposition}

\begin{corollary}\label{cor:inseperability_representatives}
The inseparability classes for \(\cO\) are in bijective correspondence with conjugacy classes of subgroups \(H\) such that \(H\to G\) is in \(\cO\). 
Moreover, if \(H\to G\) is in \(\cO\), then \(H\) is a maximal element in \([H]^{\cO}\), viewed as a subposet of \(\Sub(G)\), and all maximal elements are conjugate to \(H\). 
\end{corollary}

Inseparability has also a more intrinsic description when we use
the ghost coordinate map.

\begin{theorem}
    Two subgroups \(J\) and \(K\) are inseparable if and only if for every \(L\in \SubOG\), we have
    \[
        |(G/L)^J|=|(G/L)^{K}|.
    \]
\end{theorem}

\begin{proof}
   Let \(H=H_{\cO}(J)\). We will first show
     \[
        \big|(G/L)^J\big|=\big|(G/L)^H\big|
    \]
    for any \(L\in\SubOG\).
    Since \(J\subseteq H\), we always have
    \[
        (G/L)^J\supseteq (G/L)^H.
    \]

    For the other direction, consider an arbitrary element \(gL\in (G/L)^J\). By definition, we have
    \[
        J\subseteq gLg^{-1}.
    \]
    Since \(gLg^{-1}\to G\) is in \(\cO\) by the conjugation condition, the minimality of \(H\) implies that \(H\subseteq gLg^{-1}\), and hence \(gL\in (G/L)^H\).

    By Corollary~\ref{cor: Separates Points}, two subgroups \(H_0\) and \(H_1\) in \(\SubOG\)\index{sub@\(\SubOG\)} are conjugate if and only if the functions 
    \[
    L \mapsto \big|(G/L)^{H_0} \big| \quad \textrm{and} \quad L \mapsto \big|(G/L)^{H_1}\big|
    \]   
    from \(\SubOG\) to \(\Z\) agree. 
    Hence, \(J\) and \(K\) are inseparable if and only if the functions 
    \[
    L \mapsto \big|(G/L)^{J} \big|=\big|(G/L)^{H_{\cO}(J)} \big| \quad \textrm{and} \quad L \mapsto \big|(G/L)^{K} \big|=\big|(G/L)^{H_{\cO}(K)}\big|
    \]   
    agree, as desired.
\end{proof}

\begin{ex}
    Recall Example \ref{ex:c_6_Burnside}, where we consider \(G=C_6\). 
    The \(\cO\)-admissible orbits are the trivial ones, \(C_6/C_2\) and \(C_3/C_1\).
    Thus, by calculation of fixed points for the orbits \(C_6/C_2\) and \(C_6/C_6\) we get that \(C_6\) is inseparable from \(C_3\) and \(C_2\) is inseparable from \(C_1\) for \(\cO\).
\end{ex}

Since an \(\cO\)-admissible \(G\)-set is an \(\cO'\)-admissible \(G\)-set
whenever \(\cO \leq \cO'\), the assignment 
\[
    \cO\mapsto \mA^{\cO}(G/G)
\]
gives a functor on the poset of transfer systems. This gives us for any \(\cO\) a map of rings
\[
    \iota^{\cO}\colon \mA^{\cO}(G/G)\to \mA(G/G),
\]\index{iota@\(\iota^{\cO}\)}
and this map is induced by the inclusion of \(\cO\)-admissible \(G\)-sets into all \(G\)-sets. Combining this with the respective ghost coordinate maps gives us a square
\begin{equation}\label{eqn: Commutative Square for Ghosts}
    \xymatrix{ 
        {\mA^{\cO}(G/G)}
            \ar[r]^{\iota^{\cO}} 
            \ar[d] 
            & 
        {\mA(G/G)} 
            \ar[d] 
            \\ 
        {{\Cl^{\cO}(G;\Z)}} 
            \ar[r]_{H_{\cO}^{\ast}} 
            & 
        {{\Cl(G;\Z)}}.
        }
\end{equation}\index{Hoperad@\(H_{\cO}(-)\)}\index{class@\(\Cl^{\cO}(G;\Z)\)}

\begin{corollary}
    The square given in Equation~\ref{eqn: Commutative Square for Ghosts} commutes.
\end{corollary}

When \(\cO=\cO^{max}\), the maximal transfer system for \(G\), then the equivalence relation \(\sim_{\cO}\) is the relation of conjugacy and for every subgroup \(H\), \(e_{[H]^{\cO}}=e_{(H)}\). The commuting of the above diagram allows us to express the idempotents for \(\cO\) in terms of the ones for the terminal transfer system.

\begin{corollary}\label{cor:idempotentformula}
    For each equivalence class \([H]^{\cO}\) of inseparable subgroups, we have 
    \[
        \iota^{\cO}(e_{[H]^{\cO}})=\sum_{(K)\in [H]^{\cO}/G} e_{(K)}
    \]
    in the rational Burnside ring for \(G\), \(\mA_{\Q}(G/G)\). On the right hand side \(e_{(K)}\) denotes the idempotent in the classical rational Burnside ring Mackey functor \(\mA_{\Q}(G/G)\) corresponding to the conjugacy class of the subgroup \(K\). 
\end{corollary}

\subsection{Structure of inseparability classes}

Since 
\[
(G/H)^J=\{gH\mid g^{-1}Jg\subseteq H\}=\{gH\mid J \subseteq gHg^{-1}\},
\]
the non-emptiness of the \(J\)-fixed points of an orbit \(G/H\) is equivalent to \(J\) being subconjugate to \(H\). This gives another condition for inseparability.

\begin{proposition}\label{prop: Inseparable and subconjugacy}
	If \(J\) and \(K\) are inseparable for \(\cO\), then for all \(H\to G\) in \(\cO\), \(J\) is subconjugate to \(H\) if and only if \(K\) is subconjugate to \(H\).
\end{proposition}

\begin{lemma}
	Let \(K\) and \(J\) be subgroups that are separated by \(\cO\). If \(J\) is subconjugate to \(K\), then any subgroup of \(J\) is also separated from \(K\) by \(\cO\).
\end{lemma}
\begin{proof}
	Without loss of generality, we may assume \(J\subseteq K\). Since we assume them separated, we must have
    \[
        H_{\cO}(J)\subsetneq H_{\cO}(K).
    \]\index{Hoperad@\(H_{\cO}(-)\)}
    Proposition~\ref{prop: Retraction is monotone} then implies for any subgroup \(J_0\) of \(J\), we have
    \[
        H_{\cO}(J_0)\subseteq H_{\cO}(J)\subsetneq H_{\cO}(K),
    \]
    and thus \(J_0\) and \(K\) are separated.
\end{proof}

\begin{ex} In the result above, it is important that \(J\) is subconjugate to \(K\). To illustrate that, take \(G=C_6\) and a transfer system \(\cO\) given by identity maps, \(C_2 \to C_6\) and \(C_1 \to C_3\). Then \(C_6\) is inseparable from \(C_3\) and \(C_2\) is inseparable from \(C_1\). Also, \(C_6\) is separated from \(C_2\) and \(C_1\). Thus not every subgroup of \(C_6\) is separated from \(C_2\) (specifically the trivial subgroup), but every subgroup of \(C_2\) is separated from \(C_6\).
\end{ex}

\begin{corollary} 
	Let \(H\) be a subgroup of \(G\) such that \(H\to G\) is in \(\cO\). Then a subgroup \(K\) is in \([H]^{\cO}\) if and only if  
	\begin{enumerate}
		\item \(K\) is subconjugate to \(H\) and
		\item there is no \(J\) such that 
  \begin{enumerate}
      \item \(K\) is subconjugate to \(J\),
      \item \(J\) is subconjugate to a proper subgroup of  \(H\), and
      \item \(J\to G\) is in \(\cO\).
  \end{enumerate}  
	\end{enumerate}
\end{corollary}

Recall that a \emph{family} is a set of subgroups of \(G\) that is closed under conjugation and taking subgroups. 
Similarly, a \emph{cofamily} is the complement of a family inside the set of subgroups and \emph{relative family} is the intersection of a family with a cofamily.

\begin{corollary}\label{cor: Relative Family}
	The inseparability classes for \(\cO\) are all relative families. In particular, if \(J\in [H]^{\cO}\) and \(J\subset K\subset H\), then \(K\in [H]^{\cO}\).\index{H@\([H]^\cO\)}
\end{corollary}

In view of Corollary \ref{cor:inseperability_representatives} we will always use \(H\) such that \(H\to G\) is in \(\cO\) to represent an equivalence class \([H]^{\cO}\). We use this convention in the following surprising results.

\begin{lemma}\label{lem:Normalisers_for_Kin[H]}
    Suppose \(K\in [H]^{\cO}\) such that \(K\subseteq H\). If \(gKg^{-1}\subseteq H\) for \(g\in G\) then \(g\in N_GH\). In particular, \(N_GK\subseteq N_GH\).
\end{lemma}
\begin{proof}
 Take \(g\in G\) such that \(gKg^{-1}\subseteq H\). Since \(K\subseteq H\), \(gKg^{-1}\subseteq gHg^{-1}\). By our assumptions \(H=H_{\cO}(gKg^{-1}) \), so the uniqueness part of Proposition \ref{prop: Minimal Admissible for J} tells us that \(H=gHg^{-1}\) and therefore \(g\in N_GH\).
\end{proof}

In fact, the structure of the relative family 
\[
{[H]^{\cO}}\subseteq\Sub(G)
\]\index{sub@\(\Sub(G)\)}\index{H@\([H]^\cO\)}
is even more rigid. This is acted on by \(G\) by conjugation. Let
\[
    \Sub_{\langle H\rangle}^{\cO}=\{K\mid K\in [H]^{\cO}, K\subseteq H\}.
\]\index{sub@\(\Sub_{\langle H\rangle}^{\cO}\)}

\begin{proposition}\label{prop: tombstones}
    We have an isomorphism of \(G\)-sets
    \[
        G\timesover{N_G(H)} \Sub_{\langle H\rangle }^{\cO}\to {[H]^{\cO}}
    \]
    given by
    \[
        (g,K)\mapsto gKg^{-1}.
    \]
\end{proposition}
\begin{proof}
    It follows from Proposition \ref{prop: Minimal Admissible for J} and Lemma \ref{lem:Normalisers_for_Kin[H]} that \({[H]^{\cO}}\) together with subgroup inclusions splits into connected components, one for every conjugate of \(H\), which are permuted by \(G\). To remember the structure of this \(G\)-set it is enough to remember the connected component for \(H\), i.e. \(\{K \mid K\subseteq H \text{\ and\ } K\in [H]^{\cO} \}\) as an \(N_G(H)\)-set.
\end{proof}

\subsection{Disk-like transfer systems}\label{subsec:reducedisklike}

Recall Construction B.1.\ in \cite{RubinDetecting}, which shows how one \emph{generates} a transfer system from a set of transfers: one takes all transfers that can be obtained as a finite composition of restrictions of conjugates of the generating ones. This is the smallest transfer system that contains the starting set of transfers. 

\begin{definition}\label{def:disklike}
    We call a transfer system \(\cO\) \emph{disk-like} if the set of transfers to \(G\) in \(\cO\)
    \[
        \{H_i \to G \mid H_i \subseteq G, i \in I\}
    \]
    generates \(\cO\).
\end{definition}

Rubin showed that in this case, generation has a much simpler form.

\begin{lemma}[{\cite[Lemma A.7]{RubinDetecting}}]\label{lem: disklike generation}
    If \(\cO\) is disk-like, then for any nested subgroups \(K\subseteq H\), we have \(K\to H\) in \(\cO\) if and only if there is a subgroup \(K'\) such that
    \begin{enumerate}
        \item \(K'\to G\) is in \(\cO\) and
        \item \(K'\cap H=K\).
    \end{enumerate}
\end{lemma}

\begin{lemma}
    For any transfer system \(\cO\), the set of disk-like transfer systems which map to \(\cO\) has a maximal element \(\cO^d\)\index{operad@\(\cO^d\)}. 
\end{lemma}
\begin{proof}
    This is a direct computation: take the transfer system generated by all transfers \(H\to G\) in \(\cO\). 
\end{proof}
The maximal disk-like transfer system that includes into  \(\cO\) is the transfer system generated by \(\SubOG\). Put another way, we deduce the following result on inseparability.

\begin{proposition}
    Subgroups \(H\) and \(K\) are inseparable for \(\cO\) if and only if they are inseparable for \(\cO^d\)\index{operad@\(\cO^d\)}.
\end{proposition}

\begin{remark}
    Since the inseparability classes govern the idempotent splitting for \(\cO\)-Mackey functors, see Corollary \ref{cor:inseperability_representatives}, and 
    \[[H]:= [H]^{\cO} = [H]^{\cO^d}\]
    we will get the same list of idempotents for rational \(\cO\)-Mackey functors and rational \(\cO^d\)-Mackey functors. However, the structure of the idempotent pieces
    \[e_{[H]}\mA^\cO_\bQ \dashMod \quad \text{and} \quad e_{[H]}\mA^{\cO^d}_\bQ \dashMod \]
    might be different, as we will see in Section \ref{sec:examples} (see in particular Example \ref{ex:difference_disklike-non-disklike}).
\end{remark}

\begin{proposition}\label{prop:no_transfers_within_class_for_disklike}
    Let \(\cO\) be a disk-like transfer system for \(G\). If \(K,L \in [H]^{\cO}\)\index{H@\([H]^\cO\)}, such that \(K\subsetneq L\) then there is no transfer \(K\to L\) in \(\cO\).
\end{proposition}
\begin{proof}
    Assume that we have such \(K,L \in [H]^{\cO}\) and a transfer \(K\to L\) in \(\cO\). By Lemma~\ref{lem: disklike generation}, this transfer is a restriction of a transfer \(K' \to G\) with \(L\) and \(K'\cap L=K\). Since \(H\to G\) is in \(\cO\), by restriction \(K'\cap H \to K'\) is also in \(\cO\), and by composition so is \(K'\cap H\to G\). Note that \(K'\cap H \neq H\), since \(K'\cap L=K \neq L\). Therefore \(K \in [K'\cap H]^{\cO} \neq [H]^{\cO}\) which contradicts the fact that \(K,L \in [H]^{\cO}\).
\end{proof}

Of course, for a general transfer system \(\cO\) there might be transfers within \([H]^{\cO}\) as we illustrate below. Proposition \ref{prop: tombstones} shows that transfers within \([H]^{\cO}\) stay within one conjugate of \(H\). To be more precise, if \(J \to K\) and \(J,K \in [H]^{\cO}\), then there exists unique \(gHg^{-1}\) such that \(J,K \subseteq gHg^{-1}\).

The transfers within one inseparability class should be thought of as \emph{weak transfers}, since they do not have the strength to separate this class further. In particular, none of the transfers within an inseparability class of \([H]^{\cO}\) will end in a maximal element of \([H]^{\cO}\).

\begin{ex}\label{ex:transfers_withinClasses}
Consider \(G=C_8\) and a transfer system \(\cO\) illustrated on the picture:

\[
\xymatrix@R+0cm@C+0cm{
C_1 
\ar@/^0.pc/[r]
\ar@/_1.pc/[rr]
\ar@/_2.pc/[rrr]
&
C_2
\ar@/^0.pc/[r]
&
C_4 
&
C_8
}
\]

\vspace{1cm}

The maximal disk-like transfer system \(\cO^d \subseteq \cO\) is the following.

\[
\xymatrix@R+0cm@C+0cm{
C_1 
\ar@/^0.pc/[r]
\ar@/_1.pc/[rr]
\ar@/_2.pc/[rrr]
&
C_2
&
C_4 
&
C_8
}
\]

\vspace{1cm}

Thus, in both cases we have two inseparability classes of subgroups:
\[
[C_1]=\{C_1\}\text{ and }[C_8]=\{C_2,C_4,C_8\}.
\]

In case of the disk-like transfer system \(\cO^d\), there are no transfers within inseparability classes:

\[
\xymatrix@R+0cm@C+0cm{
[C_1] 
&
[C_2
&
C_4 
&
C_8]
}
\]

However, in the case of the transfer system \(\cO\), there is a transfer within \([C_8]\):

\[
\xymatrix@R+0cm@C+0cm{
[C_1] 
&
[C_2
\ar[r]
&
C_4 
&
C_8]
}
\]
 \end{ex}

\section{Subgroups above an inseparability class}\label{sec:transfersinsep}

\subsection{Transfers from subgroups in \texorpdfstring{\([H]^\cO\)}{[H]}}

We will now analyse the transfers in \(\cO\) between subgroups in an inseparability class \([H]^{\cO}\)\index{H@\([H]^\cO\)} and to a given subgroup \(L \subseteq G\). Recall, that in view of Corollary \ref{cor:inseperability_representatives} we will always use \(H\) such that \(H\to G\) is in \(\cO\) to represent an equivalence class \([H]^{\cO}\).

\begin{lemma}\label{lem:NON_Ab_transfer_up_to_L}
    Let \(J\in [H]^{\cO}\), and let \(L\subseteq G\) be such that we have a transfer \(J\to L\). Then
    \begin{enumerate}
        \item for all \(g\in G\), we have transfers \(L\cap gHg^{-1}\to L\),  and
        \item we have a transfer \(J\to L\cap H_{\cO}(J)\).
    \end{enumerate}
\end{lemma}

\begin{proof}
    Since \(gHg^{-1}\to G\), we have \(L\cap gHg^{-1}\to L\) by intersecting it with \(L\). Intersecting \(J\to L\) with 
    \( H_{\cO}(J)\) gives the last condition.
\end{proof}

When \(G\) is abelian, this significantly simplifies.

\begin{lemma}\label{lem:transfer_up_to_L}
    Let \(G\) be abelian, and \(J \in [H]^{\cO}\). If \(L\subseteq G\) is such that we have a transfer \(J\to L\), then
    \begin{enumerate}
        \item we have a transfer \(L\cap H\to L\) and
        \item we have a transfer \(J\to L \cap H\).
    \end{enumerate}
\end{lemma}

\begin{proposition}
    If \(K\in [H]^{\cO}\), \(K\subseteq H\) and \(K\subseteq L\), then \(H\cap L\to L\) and \(H\cap L\in [H]^{\cO}\).
\end{proposition}
\begin{proof}
    Intersecting \(H\to G\) with \(L\) gives the first part. Since \(K\subseteq H\cap L\), the second part is Corollary~\ref{cor: Relative Family}.
\end{proof}

We can restate this as saying that if \(L\) contains some subgroup in \([H]^{\cO}\), then \(L\) can be connected to \([H]^{\cO}\) by some transfer.

\begin{definition}\label{def:aboveH}
    We will say that \(L\) is \emph{above} \([H]^{\cO}\), denoted by \(L\geq [H]^{\cO}\)\index{LaboveH@\(L\geq [H]^{\cO}\)}, if there is a \(K\in [H]^{\cO}\) with \(K\subseteq L\).
\end{definition}

Lemmas \ref{lem:NON_Ab_transfer_up_to_L} and \ref{lem:transfer_up_to_L} can be reframed in terms of \([H]^{\cO}\), viewed as a subposet of \(\Sub(G)\) under \(\cO\).

\begin{corollary}\label{cor:NON_AB_weaklyTerminal}
    If \(L\geq [H]^{\cO}\), then the subposet of elements of \([H]^{\cO}\) which transfer to \(L\) is non-empty and has maximal elements
    \[
        \big\{L\cap gHg^{-1}\mid g\in G\text{ and }L\cap gHg^{-1}\in [H]^{\cO}\big\}.
    \]
\end{corollary}

\begin{corollary}\label{cor:weaklyTerminal}
    If \(G\) is abelian and \(L\geq [H]^{\cO}\), then the subposet of elements of \([H]^{\cO}\) which transfers to \(L\) has a maximum: \(L\cap H\).
\end{corollary}

The intersection of \(L\) with the conjugates of \(H\) will be used often, so we give particular notation to it.

\begin{notation}\label{not:LgH}
    If \(H,L\subseteq G\) and \(g\in G\), then let 
    \[
        L_g^{H}:=L\cap gHg^{-1}.
    \]\index{Lintersect@\(L_g^{H}\)}
    If \(H\) is clear from context, we will drop it from the notation.
\end{notation}

\begin{notation}\label{not:weakly_terminal_set}
    If \(L\geq [H]^{\cO}\), then let 
    \[
    T_{L\geq [H]^{\cO}}  = \big\{L_g^{H} \mid g\in G\text{ and }L_g^{H} \in [H]^{\cO}\big\}
    \]\index{transfer@\(T_{L\geq [H]^{\cO}}\)}
    be the set of maximal elements in the subposet of subgroups in \([H]^{\cO}\) which transfer to \(L\).
\end{notation}

\begin{proposition}\label{prop:N_GLsetT}
    Conjugation gives an action of \(N_G(L)\) on \(T_{L\geq [H]^{\cO}}\).
\end{proposition}
\begin{proof}
    Let \(n\in N_G(L)\). Since inseparable classes are conjugation invariant, if \(L_{g}^{G} \in [H]^{\cO}\), then we must have 
    \[
    n(L_{g}^{H})n^{-1}=L\cap ngHg^{-1}n^{-1}=L_{ng}^{H} \in [H]^{\cO}.
    \]
\end{proof}

The partitioning of \([H]^{\cO}\) into the conjugate pieces \(\Sub_{\langle H\rangle}^{\cO}\) also partitions the subset \(T_{L\geq [H]^{\cO}}\) of \([H]^{\cO}\).

\begin{proposition}\label{prop: TLgeqH maps to GmodNH}
    Each conjugate of \(\Sub_{\langle H\rangle}^{\cO}\)\index{sub@\(\Sub_{\langle H\rangle}^{\cO}\)} contains at most one element of \(T_{L\geq [H]^{\cO}}\)\index{transfer@\(T_{L\geq [H]^{\cO}}\)}. 
\end{proposition}

Using the unexpected functoriality of the normalizers for subgroups in \([H]^{\cO}\), we can prove a stronger statement.

\begin{proposition}\label{prop: Normalizer in L of LgH}
    If \(L\geq [H]^{\cO}\)\index{LaboveH@\(L\geq [H]^{\cO}\)} and \(L_g^H\in T_{L\geq [H]^{\cO}}\), then 
    \[
        N_L(L_g^H)=N_G(gHg^{-1})\cap L.
    \]
\end{proposition}
\begin{proof}
    By construction, we have
    \[
        H_{\cO}(L_g^{H})=gHg^{-1},
    \]
    so by Lemma~\ref{lem:Normalisers_for_Kin[H]}, we have
    \[
        N_G(L_g^H)\subseteq N_G(gHg^{-1}).
    \]
    Intersecting this with \(L\) gives
    \[
        N_L(L_g^{H})=L\cap N_G(L_g^H)\subseteq L\cap N_G(gHg^{-1}).
    \]
    On the other hand, if \(\ell\in L\cap N_G(gHg^{-1})\), then \(\ell\) normalizes both \(L\) and \(gHg^{-1}\). Hence it normalizes their intersection, giving the other inclusion.
\end{proof}

When \(G\) is abelian and \(\cO\) is disk-like, then we can make a further simplification.

\begin{corollary}
    Suppose \(G\) is an abelian group. If \(\cO\) is a disk-like transfer system then for every \(L\geq [H]^{\cO}\), there is a unique element \(J\in [H]^{\cO}\) such that there is a transfer \(J \to L\).
\end{corollary}
\begin{proof}
    This follows from Lemma \ref{lem:transfer_up_to_L}, Corollary \ref{cor: Relative Family} and Proposition \ref{prop:no_transfers_within_class_for_disklike}.
\end{proof}

\begin{remark}
Note that being above \([H]^{\cO}\) does not translate well to inseparability classes, that is if \(K\geq [H]^{\cO}\) and \(L\) is inseparable from \(K\), it does not imply that \(L\) is above \([H]^{\cO}\), see Part 6 of Example \ref{ex:list_for_C6}. There \(C_6\) is above \([C_3]\), but \(C_2\), which is inseparable from \(C_6\), is not above \([C_3]\).
\end{remark}

Now consider the situation of two subgroups \(K,L\), such that \(K\subseteq L\) and \(K,L \geq [H]^{\cO}\).

\begin{proposition}\label{prop:Claim2}
    Suppose \(K\subseteq L\) and \(K \geq [H]^{\cO}\). For every \(g\in G\) such that \(K_g^{H}\in[H]^{\cO}\) we have
    \begin{enumerate}
        \item \(L_g^{H}\in[H]^{\cO}\), and 
        \item \(L_{g}^{H}\)\index{Lintersect@\(L_g^{H}\)} is the unique subgroup in the set \(T_{L\geq [H]^{\cO}}\)\index{transfer@\(T_{L\geq [H]^{\cO}}\)} that contains \(K_g^{H}\).
    \end{enumerate}
\end{proposition}
\begin{proof}
    It is clear that we have
    \[
    K_{g}^{H}\subseteq L_{g}^{H}\subseteq gHg^{-1},
    \]
    and hence Corollary~\ref{cor: Relative Family} gives the first part. The uniqueness follows from Proposition~\ref{prop: TLgeqH maps to GmodNH}.
\end{proof}
    
\begin{corollary}
    If \(K\geq [H]^{\cO}\) and \(K\subseteq L\), then the assignment
    \[
        K_{g}^{H}\mapsto L_{g}^{H}
    \]\index{Lintersect@\(L_g^{H}\)}
    gives an injective function
    \[
        \iota\colon T_{K\geq [H]^{\cO}}\to T_{L\geq [H]^{\cO}}.
    \]\index{transfer@\(T_{L\geq [H]^{\cO}}\)}
    Moreover, this function is \(K\)-equivariant for the conjugation action.
\end{corollary}

\begin{proof}
    The only thing to show is the equivariance. Here, we use that \(K\) is contained in the normalizers in \(G\) both of itself and of \(L\). The result follows from uniqueness.
\end{proof}

\begin{proposition}\label{prop:Claim1}
    Suppose \(K\subseteq L\) and \(K,L \geq [H]^{\cO}\).\index{LaboveH@\(L\geq [H]^{\cO}\)} Then for every \(g\in G\) such that \(K_{g}^{H} \in [H]^{\cO}\) we have \(N_K(K_{g}^{H}) \subseteq N_L(L_{g}^{H})\).
\end{proposition}
\begin{proof}
By Propositions \ref{prop:Claim2} and \ref{prop: Normalizer in L of LgH} we have \(N_K(K_g^H)=N_G(gHg^{-1})\cap K\) and \(N_L(L_g^H)=N_G(gHg^{-1})\cap L\), thus \(N_K(K_g^H) \subseteq N_L(L_g^H)\).
\end{proof}

\subsection{Idempotents and restrictions}

As well as understanding how transfers from an inseparability class behave, 
we need to understand how restrictions interact with our idempotents.

\begin{lemma}\label{lem:idempotentrestrictzero}
Let \(H \to G \in \cO\). The idempotent
\(e_{[H]^{\cO}} \in {\mA^{\cO}_{\bQ}(G/G)}\)\index{e@\(e_{[H]^{\cO}}\)}
restricts to zero in \({\mA^{\cO}_{\bQ}(G/K)}\)
whenever \(K\) is not above \([H]^{\cO}\).
\end{lemma}
\begin{proof}
    Let \(K\) be a subgroup of \(G\) that is not above \([H]^{\cO}\). 
    Consider the commutative diagram below, where the vertical
    maps are the restrictions. 
    \[
    \xymatrix{
        {\mA^{\cO}_{\bQ}(G/G)} 
            \ar[r]^\cong 
            \ar[d] 
            &
        {{\Cl^{\cO}(G;\Q)}}  
            \ar[r]^{H_{\cO}^{\ast}} 
            &
        {{\Cl(G;\Q)}}
            \ar[d] 
            \\
        {\mA^{\cO}_{\bQ}(G/K)} 
            \ar[r]^\cong 
            &
        {{\Cl^{\cO}(K;\Q)}}  
            \ar[r]_{H_{\cO}^{\ast}} 
            &
        {{\Cl(K;\Q)}}
    }
    \]
    By Corollary \ref{cor:idempotentformula}
    \(e_{[H]^{\cO}} \in \mA^{\cO}_{\bQ}(G/G)\) is sent to 
    \[
    \sum_{(K)\in [H]^{\cO}/G} e_{(K)} \in \Cl(G;\Q)
    \]
    which restricts to zero in \({{\Cl(K;\Q)}}\) as
    \(\Sub(K)\) contains no element of \([H]^{\cO}\).

    As the lower horizontal maps are injections (see Corollary \ref{cor: Class Functions Split}) 
    it follows that \(e_{[H]^{\cO}}\) restricts to zero
    in  \({\mA^{\cO}_{\bQ}(G/K)}\).
\end{proof}
We invite the reader to prove this result directly from Lemma \ref{lem:idempotentformula}.

\begin{corollary}\label{cor: Idempotents are Above H}
    If \(L\) is not above \([H]^{\cO}\), then for any rational \(\cO\)-Mackey functor \(\mM\), we have
    \[
        e_{[H]^\cO}\mM(G/L)=0.
    \]
\end{corollary}

\section{Mackey functors on and above an inseparability class}\label{sec:description_of_idempotent_pieces}

By Corollary~\ref{cor: Idempotents are Above H}, the idempotent summand \(e_{[H]^{\cO}}\mM\) of a rational \(\cO\)-Mackey functor \(\mM\) vanishes on any orbit \(G/L\) with \(L\) not above \([H]^{\cO}\). We analyze the category of \(\cO\)-Mackey functors concentrated above \([H]^{\cO}\) and determine the essential image of the idempotent localization \(e_{[H]^{\cO}}\).

\begin{definition}
Define a category \emph{above-\([H]^{\cO}\)-\(\Mackey^G\)} to be the full subcategory of \(\cO\)-Mackey functors spanned by those \(\mM\) such that \(\mM(G/K)=0\) unless \(K \geq [H]^{\cO}\). We denote this category by \(\AboveHMackey[{[H]^{\cO}}]^G\)\index{above@\(\AboveHMackey[{[H]^{\cO}}]^G\)}. 
\end{definition}

\subsection{\texorpdfstring{\([H]^\cO\mhyphen\)}{[H]-}Mackey functors}

\begin{definition}\label{def:A_(H)}
Let \(\Fin^{G}_{(H)}\)\index{finite@\(\Fin^{G}_{(H)}\)} denote the full subcategory of \(\Fin^G\)\index{finite@\(\Fin^G\)} spanned by those finite \(G\)-sets \(T\) such that for all \(t\in T\), \(\Stab(t)\) is subconjugate to \(H\).

Let \(\cA^{\cO}_{(H)}\)\index{A@\(\cA^{\cO}_{(H)}\)} denote the full subcategory of \(\cA^{\cO}\)\index{A@\(\cA^{\cO}\)} spanned by objects \(T\in\Fin^{G}_{(H)}\).
\end{definition}

\begin{remark}\label{rem:structure_of_(H)-Mackey}
Given any \(H\leq G\), we have an associated family of subgroups, \(\mathcal F(H)\), namely the family of subgroups of \(G\) subconjugate to \(H\). A family of subgroups is the same data as a sieve in the orbit category, and these are necessarily closed under pullbacks. This shows that \(\cA^{\cO}_{(H)}\) is equivalently the ordinary \(\cO\)-Lindner category for \(\Fin^{G}_{(H)}\).
\end{remark}

\begin{definition}
We define the category of \(\cO\)-Mackey functors on the family \(\mathcal F(H)\) to be the category of product preserving functors
\[
    \cA^{\cO}_{(H)}\to \Ab.
\]
Denote this category by \(\OMackey^{G}_{\mathcal F(H)}\)\index{omackey@\(\OMackey^G_{\mathcal F(H)}\)}.
\end{definition}

\begin{remark}
    The category \(\OMackey^{G}_{\mathcal F(H)}\) is in general {\emph{not}} equivalent to the expected \(i_H^\ast\OMackey^{H}\). The category of \(\cO\)-Mackey functors on \(\mathcal F(H)\) has conjugation maps for arbitrary elements of \(G\), so in particular, we see an equivariance for the normalizer of \(H\) in \(G\) that we do not see for a general \(i_H^\ast\cO\)-Mackey functor for \(H\).
\end{remark}

Note that we have a natural fully-faithful inclusion
\[
    \iota\colon \cA^{\cO}_{(H)}\hookrightarrow \cA^{\cO}
\]
that commutes with the categorical product. This gives a restriction on Mackey functors.

\begin{proposition}\label{prop:restrictiontriple}
    The restriction functor
    \[
    \restrict{}{(H)}\colon \OMackey^{G}\to\OMackey^{G}_{\mathcal F(H)}
    \]
    given by precomposition with \(\iota\) has a left adjoint \(\IndH[(H)]\) given by the coend
    \[
        \IndH[(H)]\mM = \int^{G/K \in \cA^{\cO}_{(H)}} \cA^{\cO} (G/K, -) \otimes \mM(G/K)
    \]
    and a right adjoint \(\CoIndH[(H)]\) given by the end
    \[
        \CoIndH[(H)]\mM =  \int_{G/K \in \cA^{\cO}_{(H)}} \hom\big( \cA^{\cO} (-, G/K), \mM(G/K) \big)
    \]
    Moreover, for \(\mM \in \OMackey^{G}_{\mathcal F(H)}\) the unit 
    \[
        \mM \longrightarrow \restrict{}{(H)} \IndH[(H)] \mM
    \]
    and the counit 
    \[
        \restrict{}{(H)} \CoIndH[(H)] \mM \longrightarrow \mM
    \]
    are isomorphisms.
\end{proposition}
\begin{proof}
The coends and ends (which are the left and right Kan extensions along \(\iota\)) 
exist since the category of abelian groups is complete and cocomplete.
These formulas  give the adjoints as the restriction functor is 
induced by a map of diagrams \(\cA^{\cO}_{(H)}\hookrightarrow \cA^{\cO}\).

The last statement is a consequence of the inclusion
\[
    \iota\colon\cA^{\cO}_{(H)}\to\cA^{\cO}
\]
being fully-faithful.
\end{proof}

\begin{remark}
    Induction and coinduction here are the ones for extending a Mackey functor defined on a family of subgroups to a Mackey functor on all subgroups of \(G\). Again, we also record the data of the conjugation action by arbitrary elements of \(G\).
\end{remark}

Restricting \(\AboveHMackey[{[H]^{\cO}}]^G\) down to the family of subgroups subconjugate to \(H\), we introduce our main category of interest.

\begin{definition}\label{def:[H]-Mackey}
    Let \([H]^{\cO}\mhyphen\Mackey^{G}\)\index{Hmackey@\([H]^\cO\mhyphen\Mackey^G\)} denote the full subcategory of \(\cO\)-Mackey functors on \(\mathcal F(H)\) spanned by those \(\mM\) such that if \(K\not\in [H]^{\cO}\), then \(\mM(G/K)=0\).
\end{definition}

An \([H]^{\cO}\)-Mackey functor \(\mM\) consists of a value \(\mM(G/K)\) for any subgroup \(K\) subconjugate to \(H\) which is \(0\) for all \(G/L\) where \(L \not \in [H]^{\cO}\). Moreover, it has restriction, additive transfers and conjugation maps encoded by \(\cA_{(H)}^{\cO}\)\index{A@\(\cA^{\cO}_{(H)}\)}.

\begin{theorem}\label{thm:res_ind_coind_exist}
    The  restriction functor and both adjoints restrict to give adjoint pairs of functors
    \[
    \xymatrix@C-0.3cm{
       \restrict{}{[H]^{\cO}} \colon {\AboveHMackey[{[H]^{\cO}}]}^G  
        \ar@<-0.5ex>[r]
        &
        \ar@<-0.5ex>[l]
        [H]^{\cO}\mhyphen\Mackey^G\ \colon \IndHO,
    }
    \]
    \[
    \xymatrix@C-0.3cm{
       \restrict{}{[H]^{\cO}} \colon \AboveHMackey[{[H]^{\cO}}]^{G}  
        \ar@<+0.5ex>[r]
        &
        \ar@<+0.5ex>[l]
        [H]^{\cO}\mhyphen\Mackey^G\ \colon \CoIndHO.
    }
    \]\index{Ind@\(\IndHO\)}\index{restrict@\(\restrict{}{[H]^{\cO}}\)}\index{coind@\(\CoIndHO\)}

\end{theorem}
\begin{proof}
    The definition of the category \([H]^{\cO}\mhyphen\Mackey^G\) shows that the image of the restriction functor from \(\AboveHMackey[{[H]^{\cO}}]^G\) lands in it. For the remaining results, since these are full subcategories, we need only to show that the left and right adjoints land in the full subcategory \(\AboveHMackey[{[H]^{\cO}}]^G\). 
    
    For the left adjoint, consider some 
    \(\mM \in [H]^{\cO}\mhyphen\Mackey^G\). By construction, we have
    \[
        \IndHO\mM (G/L) = \int^{G/K \in \cA^{\cO}_{(H)}} \cA^{\cO} (G/K, G/L) \otimes M(G/K).
    \]
    We need to check that if \(L\not\geq [H]^\cO\), then this is zero.

    Consider any span \(\alpha\) from \(G/K\) to \(G/L\) of the form
    \[
        G/K\xleftarrow{f_R} G/J\xrightarrow{f_T} G/L.
    \] 
    Since \(J\) is subconjugate to \(L\), if \(L\) is not above \([H]^\cO\), then \(J\) must also not be above \([H]^\cO\). Moreover, since \(G/J\to G/K\) and \(G/K\in \cA^{\cO}_{(H)}\), we deduce that \(J\) is subconjugate to \(H\). This means that the map \(G/J\to G/K\) is in \(\cA^{\cO}_{(H)}\), and we can factor \(\alpha\) as
    \[
        \big(G/J\xleftarrow{=} G/J\xrightarrow{f_T} G/L\big)\circ \big(G/K\xleftarrow{f_R} G/J\xrightarrow{=} G/J\big).
    \]
    In the coend the term corresponding to 
    \[
    \alpha \otimes m \in \cA^{\cO} (G/K, G/L) \otimes M(G/K)
    \]
    is identified with 
    \[
        f_T\otimes f_{R}^{\ast}(m)\in \cA^{\cO} (G/J, G/L) \otimes \mM(G/J).
    \]
    Since \(J\) is not in \([H]^\cO\) but is in the family of subgroups subconjugate to \(H\), \(\mM(G/J)=0\) by assumption.
    This gives the claim. A similar argument holds for the right adjoint.  
\end{proof}

\begin{corollary}\label{cor:indH is fully faithful}   
For \(\mM\) an \([H]^\cO\)-Mackey functor, the unit 
    \[ 
        \eta_{\mM}\colon \mM \to \restrict{}{[H]^\cO} \IndHO\mM 
    \]  and counit 
    \[
        \epsilon_{\mM}\colon \CoIndHO \restrict{}{[H]^\cO} \mM \to \mM 
    \]
    are isomorphisms. \index{Ind@\(\IndHO\)}\index{restrict@\(\restrict{}{[H]^{\cO}}\)}\index{coind@\(\CoIndHO\)}
\end{corollary}

\begin{notation}
    Since \(\IndHO\) is fully-faithful, we will identify \([H]^{\cO}\mhyphen\Mackey^G\) with the image under induction in \(\AboveHMackey[{[H]^{\cO}}]^G\).
\end{notation}

\subsection{Analyzing the left adjoint}

\subsubsection{Values and conjugations}

\begin{proposition}\label{prop:Ind}
For any \([H]^{\cO}\)-Mackey functor \(\mM\), \(\IndHO\mM \in \AboveHMackey[{[H]^{\cO}}]^G\)\index{above@\(\AboveHMackey[{[H]^{\cO}}]^G\)} satisfies the formula 
    \[
        \IndHO\mM(G/L)=\begin{cases}
           0 & L\not \geq [H]^{\cO} \\
           \displaystyle\bigoplus_{L_g^{H}\in T_{L\geq [H]^{\cO}}/L} \mM(G/(L_{g}^{H}))_{W_L (L_{g}^{H})} & L\geq [H]^{\cO}.
            \end{cases}
    \]\index{transfer@\(T_{L\geq [H]^{\cO}}\)}\index{Ind@\(\IndHO\)}
\end{proposition}

In our analysis that follows, we will use the coend formula for the left adjoint: \(\IndHO\mM(G/L)\) is the coequalizer of the diagram
\[
\xymatrix@C-0.3cm{
\displaystyle\bigoplus_{K,J\in\mathcal F(H)} \mM(G/J)\otimes\cA^{\cO}_{(H)}(G/J,G/K)\otimes\cA^{\cO}(G/K,G/L)
\ar@<+0.5ex>[r]
\ar@<-0.5ex>[r]
& 
\displaystyle\bigoplus_{K\in\mathcal F(H)} \mM(G/K)\otimes \cA^{\cO}(G/K,G/L)
}
\]

\begin{proof} 
    The Weyl group \(W_L (L_{g}^{H})\) is a subgroup of \(W_G (L_{g}^{H})\) thus the orbits are taken with respect to the usual action. We may view this as an action of the normaliser \(N_L (L_{g}^{H})\), rather than the Weyl group, via the usual quotient map.

    The coend formula for \(\IndHO\mM(G/L)\) simplifies in two ways. Since \(\mM\) is assumed to be in the category \([H]^{\cO}\mhyphen\Mackey^G\), if \(K\) is not in \([H]^{\cO}\), then \(\mM(G/K)=0\). This immediately allows us to throw out any summands indexed by subgroups other than those in \([H]^{\cO}\). 
    
    Now consider \(K\in [H]^{\cO}\), and assume we have a morphism in \(\cA^{\cO}(G/K,G/L)\) of the form
    \[
        G/K\xleftarrow{f_R} G/J\xrightarrow{f_T}G/L.
    \]
    If \(J\not\in [H]^{\cO}\), then this is a composite
    \[
        \big(G/\big(H_{\cO}(J)\cap L\big)\xleftarrow{=} G/\big(H_{\cO}(J)\cap L\big) \xrightarrow{} G/L\big)\circ \big(G/K\xleftarrow{f_R} G/J\xrightarrow{} G/\big(H_{\cO}(J)\cap L\big)\big).
    \]
    The span 
    \[
    G/K\xleftarrow{f_R} G/J\xrightarrow{} G/\big(H_{\cO}(J)\cap L\big)
    \]
    is in \(\cA^{\cO}_{(H)}\). If \(J\not\in [H]^{\cO}\), then both \(H_{\cO}(J)\) and \(H_{\cO}(J)\cap L\) are not in \([H]^{\cO}\). This means that any element of the form
    \[
        m\otimes \big(G/K\xleftarrow{f_R} G/J\xrightarrow{f_T}G/L\big)
    \]
    with \(J\not\in[H]^{\cO}\) is also set equal to zero.

    Finally, given a span 
    \[
        G/K\xleftarrow{f_R} G/J\xrightarrow{f_T}G/L
    \]
    with both \(K\) and \(J\) in \([H]^{\cO}\). By Lemma~\ref{lem:NON_Ab_transfer_up_to_L}, we know this can be factored as a composite of
    \[
        (G/L_g^{H} \xleftarrow{=} G/L_g^H\to G/L)\circ(G/K\xleftarrow{f_R} G/J\xrightarrow{\tilde{f}_T} G/L_g^{H})
    \]
    where \(g\) is such that \(J\subseteq gHg^{-1}\). The span
    \[
        G/K\xleftarrow{f_R} G/J\xrightarrow{\tilde{f}_T} G/L_g^{H}
    \]
    is in \(\cA^{\cO}_{(H)}\), so in the coend formula, any element 
    \[
        m\otimes\alpha\in \mM(G/K)\otimes\cA^{\cO}(G/K,G/L)
    \]
    is canonically identified with
    \[
        \big(f_{T\ast} f_{R}^{\ast}(m)\big)\otimes \big(G/L_g^{H} \xleftarrow{=} G/L_g^H\to G/L\big).
    \]
    This means we can rewrite the coend as a quotient of 
    \[
        \bigoplus_{L_g^H\in T_{L\geq [H]^{\cO}}} \mM(G/L_g^H).
    \]
    An element \(m\) in the summand corresponding to \(L_g^H\) represents the simple tensor
    \[
        m\otimes (G/L_g^H\xleftarrow{=} G/L_g^H \to G/L),
    \]
    where the right-hand map is the canonical quotient. 

    The ``conjugation'' action in the Mackey functor arises as restriction along the isomorphism
    \[
        a\colon G/aLa^{-1}\to G/L
    \]
    given by 
    \[
        gaLa^{-1}\mapsto gaL.
    \]
    When we compose the span \(G/L_g^H\xleftarrow{=} G/L_g^H \to G/L\) with this, then we get the span
    \[
        G/L_g^{H} \xleftarrow{a} G/aL_g^Ha^{-1} \to G/aLa^{-1},
    \]
    which factors as
    \[
        (G/aL_g^Ha^{-1}\xleftarrow{=} G/aL_g^Ha^{-1}\to G/aLa^{-1})\circ (G/L_g^{h}\xleftarrow{a} G/aL_g^Ha^{-1}\xrightarrow{=}G/aL_g^Ha^{-1}).
    \]
    The part involving just isomorphisms is the usual conjugation, viewed as being in \(\cA^{\cO}_{(H)}\).

    Now, if \(a\in L\), then the induced map \(a\colon G/L\to G/L\) is the identity. If \(aL_g^H a^{-1}\neq L_g^H\), then the coend identifies summands, via \(c_a^{\ast}\), further simplifying the sum to 
    \[
        \bigoplus_{L_g^H\in T_{L\geq [H]^{\cO}}/L} \mM(G/L_g^H).
    \]
    Finally, if \(a\in N_{L}(L_g^H)\), then the identity and \(c_{a}^\ast\) are coequalized, giving as a final answer
    \[
        \bigoplus_{L_g^H\in T_{L\geq [H]^{\cO}}/L} \mM(G/L_g^H)_{N_{L}(L_g^H)}.
    \]
    The last term can further be identified with the quotient by the Weyl action.

    The summands indexed by \(L_{g}^{H}\) and 
    \[
        \ell(L_{g}^{H})\ell^{-1}=\ell gHg^{-1}\ell^{-1}\cap L,
    \] 
    are identified since conjugation by \(\ell\in L\) gives an isomorphism in \(G\)-sets over \(G/L\): 
    \[
        G/(L_{g}^{H}) \cong G/\ell(L_{g}^{H})\ell^{-1}.
    \]

    Next, each of the summands is then quotiented by the automorphisms in \(G\)-sets over \(G/L\) of \(G/(L_{g}^{H})\), which is 
    \[
        \big(N_G(gHg^{-1})\cap L\big)/ (L_{g}^{H}).
    \] 
    Finally, this group can be identified with the corresponding Weyl group by Proposition~\ref{prop: Normalizer in L of LgH}:
    \[W_L (L_{g}^{H})=N_L(L_{g}^{H})/(L_{g}^{H})=\big(N_G(gHg^{-1})\cap L\big)/ (L_{g}^{H}). \qedhere
    \]
\end{proof}

One way to interpret the value of \(\IndHO\mM\) at \(G/L\) for \(L\geq [H]^{\cO}\) is to note that it is the sum of formal transfers from the maximal subgroups of \([H]^{\cO}\) which transfer up to \(L\). The relations then are simply the usual equivariance relations for the transfer and how it connects to conjugations.

In our proof of Proposition~\ref{prop:Ind}, we used a description of the conjugation action. We record it here as well.

\begin{notation}
    For \(m\in \mM(G/L_g^H)\), let \(\indtr_{L_g^H}^{L}(m)\)\index{transfer@\(\indtr_{L_g^H}^{L}(m)\)} denote the image of the simple tensor
    \[
        m\otimes (G/L_g^H\xleftarrow{=} G/L_g^H \to G/L).
    \]
\end{notation}

\begin{proposition}
    For a \([H]^{\cO}\)-Mackey functor \(\mM\), the conjugation by \(a\) map
    \[
        c_a^{\ast}\colon \IndHO\mM(G/L)\to \IndHO\mM(G/aLa^{-1})
    \]
    is determined by the formula
    \[
        c_a^{\ast} \indtr_{L_g^H}^L(m) = \indtr_{aL_g^Ha^{-1}}^{aLa^{-1}} c_a^{\ast}(m),
    \]
    where the right most \(c_a^{\ast}(m)\) is evaluated in the \([H]^{\cO}\)-Mackey functor \(\mM\).
\end{proposition}

\subsubsection{Restrictions} Suppose \(K\subseteq L\) and both \(K,L \geq [H]^{\cO}\). 

\begin{proposition}\label{prop:restrictionequation}
    If \(K\geq [H]^{\cO}\) and \(K\subseteq L\), then the restriction map
    \[
        \res_{K}^{L}\colon \IndHO\mM(G/L)\to \IndHO\mM(G/K)
    \]
    is given by the formula
    \[
        \res_{K}^{L} \indtr_{L_g^H}^L(m)=\sum_{\ell\in L_g^H\backslash L/K} 
        \indtr_{K_{\ell^{-1}g}^H}^{K} c_{\ell}^{\ast} \res_{L_g^H\cap \ell K\ell^{-1}}^{L_g^H}(m).
    \]
\end{proposition}
\begin{proof}
    By definition, the restriction to \(K\) of the element \(\indtr_{L_g^H}^{L}(m)\) is represented in the coend by
    \[
        m\otimes \big(G/L_g^H\leftarrow G/L_g^H\timesover{G/L} G/K\to G/K\big).
    \]

    The result follows from the identification
    \[
        G/L_g^H\timesover{G/L} G/K\cong G\timesover{L} \big(L/L_g^H\times L/K\big)
    \]
    and applying the ordinary double coset formula.
\end{proof}

\subsubsection{Transfers}

\begin{proposition}
    If \(K\geq [H]^{\cO}\) and \(K \to L\), then the transfer map
    \[
        \tr_{K}^{L}\colon \IndHO\mM(G/K)\to \IndHO\mM(G/L)
    \]
    is given by the formula
    \[
        \tr_{K}^{L} \indtr_{K_g^H}^{K}(m)= \indtr_{L_g^H}^K \tr_{K_g^H}^{L_g^H}(m).
    \]
\end{proposition}
\begin{proof}
    The transfer to \(L\) of the element \(\indtr_{K_g^H}^{K}(m)\) is represented by 
    \[
        m\otimes \big(G/K_g^H\xleftarrow{=} G/K_g^H\to G/L\big). \qedhere
    \]
\end{proof}

This gives us a very important structural result about induction.

\begin{corollary}\label{cor:transfers_generate_value_Ind}
    If \(\mM\) is in \([H]^{\cO}\mhyphen\Mackey^G\), and \(L\) is above \([H]^{\cO}\), then   
    \[
        \IndHO \mM(G/L)
    \]
    is generated by transfers from 
    \[
        \mM(G/K)=\restrict{}{[H]^{\cO}}\IndHO \mM(G/K)
    \]
    for \(K\in [H]^{\cO}\).
\end{corollary}

\begin{corollary}
 If \(G\) is an abelian group then 
    for any \([H]^{\cO}\)-Mackey functor \(\mM\), we have the formula 
   \[ \IndHO\mM(G/L)= \mM(G/H\cap L)_{L/H\cap L}. 
   \] 
 Notice that this formula gives value \(0\) for all \(L\not\geq [H]^{\cO}\).
\end{corollary}

An extremely important consequence of Corollary~\ref{cor:transfers_generate_value_Ind} is that we also have a very simple formula for the counit of the adjunction.

\begin{proposition}\label{prop: Indentifying the maps in the counit}
    For any \(\mM\) in \(\AboveHMackey[{[H]^{\cO}}]^G\)\index{above@\(\AboveHMackey[{[H]^{\cO}}]^G\)} and for any \(L\geq [H]^{\cO}\)\index{LaboveH@\(L\geq [H]^{\cO}\)}, the counit
    \[
        \bigoplus_{L_g^H\in\mathcal T_{L\geq [H]^{\cO}}/L}\big(\mM(G/L_g^H)\big)_{W_L(L_g^H)}\cong
        \IndHO\restrict{}{[H]^{\cO}}\mM(G/L)\to \mM(G/L)
    \]    
    is just the sum of the transfer maps
    \[
        tr_{L_g^H}^{L}\colon \big(\mM(G/L_g^H)\big)_{W_L(L_g^H)}\to \mM(G/L).
    \]
\end{proposition}
\begin{proof}
    The counit is a map of Mackey functors, and the element denoted \(\indtr_{L_g^H}^L(m)\) in the source is the transfer of \(m\in \mM(G/L_g^H)\). The result follows from the recollection that for every \(K\in [H]^{\cO}\), the map
    \[
        \IndHO\restrict{}{[H]^{\cO}}\mM(G/K)\to \mM(G/K)
    \]
    is a canonical isomorphism taking \(m\) to itself.
\end{proof}
  
\subsection{Analyzing the right adjoint}
A similar proof to Proposition \ref{prop:Ind} gives an identification of the coinduction functor. As in that result, the end formula decomposes into an indexed sum over spans of the form
\(G/L \leftarrow G/L_g^H \xrightarrow{=} G/L_g^H\) where \(L_{g}^{H}\in T_{L\geq [H]^{\cO}}/L\).\index{transfer@\(T_{L\geq [H]^{\cO}}\)}

\begin{proposition}\label{prop:Coind}
    For any \([H]^{\cO}\)-Mackey functor \(\mM\), \(\CoIndHO\mM \in \AboveHMackey[{[H]^{\cO}}]^G\)\index{above@\(\AboveHMackey[{[H]^{\cO}}]^G\)} satisfies the formula
    \[
        \CoIndHO\mM(G/L)=\begin{cases}
        0 & L\not\geq [H]^{\cO} \\
         \displaystyle\prod_{L_{g}^{H}\in T_{L\geq [H]^{\cO}}/L} \mM(G/(L_{g}^{H}))^{W_L(L_{g}^{H})} & L\geq [H]^{\cO}.
        \end{cases}
    \]\index{transfer@\(T_{L\geq [H]^{\cO}}\)}\index{coind@\(\CoIndHO\)}
\end{proposition}

\begin{notation}
    For \(m\in \mM(G/L_g^H)^{W_L(L_{g}^{H})}\) and \(L_{g}^{H}\in T_{L\geq [H]^{\cO}}/L\), 
    write \(\cotr_{L_g^H}^{L}(m)\)\index{cotransfer@\(\cotr_{L_g^H}^{L}(m)\)} for the element which in summand 
    \(L_{g}^{H}\) is \(m\).  
\end{notation}

We think of \(\cotr_{L_g^H}^{L}(m)\) as a formal co-transfer of \(m\).
In terms of the end formula, it corresponds to the map which sends the span
\(G/L \leftarrow G/L_g^H \xrightarrow{=} G/L_g^H\) to \(m\) and is zero elsewhere. 
As the end formula considers maps out of the Lindner category, we regard
the span \(G/L \leftarrow G/L_g^H \xrightarrow{=} G/L_g^H\) 
not as a restriction, but as a transfer and use the term co-transfer
for clarity. 
This interpretation allows us to 
describe conjugations, restrictions and transfers for \( \CoIndHO\mM\).

\begin{proposition}
    For a \([H]^{\cO}\)-Mackey functor \(\mM\), the conjugation by \(a\) map
    \[
        c_a^{\ast}\colon \CoIndHO\mM(G/L)\to \CoIndHO\mM(G/aLa^{-1})
    \]
    is determined by the formula
    \[
        c_a^{\ast} \cotr_{L_g^H}^L(m)=\cotr_{a(L_g^H)a^{-1}}^{aLa^{-1}} c_a^{\ast}(m),
    \]\index{cotransfer@\(\cotr_{L_g^H}^{L}(m)\)}
    where the right most \(c_a^{\ast}(m)\) is evaluated in the \([H]^{\cO}\)-Mackey functor \(\mM\).
\end{proposition}

We give descriptions of the restriction and transfer maps between the values at two subgroups which are both above \([H]^{\cO}\).

\begin{proposition}
    If \(K\geq [H]^{\cO}\) and \(K\subseteq L\), then the restriction map
    \[
        \res_{K}^{L}\colon \CoIndHO\mM(G/L)\to \CoIndHO\mM(G/K)
    \]
    is given by the formula
    \[
    \res_{K}^{L} \cotr_{L_g^H}^{L}(m)= \cotr_{K_g^H}^K \res_{K_g^H}^{L_g^H}(m).
    \]\index{cotransfer@\(\cotr_{L_g^H}^{L}(m)\)}
\end{proposition}
\begin{proof}
    We describe the restriction to \(K\) of the element \(\cotr_{L_g^H}^{L}(m)\). 
    In the end description it is constructed from precomposition 
    by the restriction from \(L\) to \(K\).  
    Hence, in terms of the end formula it is the map on spans   
    \[
    \begin{array}{rcl}
        \big( G/K \leftarrow G/K_g^H \xrightarrow{=} G/K_g^H \big) 
        & \mapsto & 
        \big( G/L \leftarrow G/K_g^H \xrightarrow{=} G/K_g^H \big) \\
        & = & 
        \big( G/L_g^H \leftarrow G/K_g^H\to G/K_g^H \big) 
        \circ
        \big( G/L \leftarrow G/L_g^H\to G/L_g^H \big)
    \end{array}
    \] 
    combined with the map which sends \( G/L \leftarrow G/L_g^H\to G/L_g^H \) to \(m\).
    The end equalises post-composition with spans in \(\cA^{\cO}_{(H)}\)
    with the action of the spans on the Mackey functor. 
    Hence, it is the map   
    \[
        \big( G/K \leftarrow G/K_g^H\to G/K_g^H \big) 
        \mapsto 
        \res_{K_g^H}^{L_g^H}(m).    \qedhere
    \]
\end{proof}

\begin{proposition}
    If \(K\geq [H]^{\cO}\) and \(K \to L\), then the transfer map
    \[
        \tr_{K}^{L}\colon \CoIndHO\mM(G/K)\to \CoIndHO\mM(G/L)
    \]
    is given by the formula
        \[
        tr_K^L \cotr_{K_g^H}^K (m)
        =
        \sum_{\ell\in K\backslash L/L_g^H} 
        \cotr_{L_{\ell^{-1} g}^H}^L  
        \left(
        \tr_{(\ell^{-1} K\ell)_{\ell^-1 g}^{H}}^{L_{\ell^{-1} g}^H}  
        c_{\ell}^{\ast}       
        (m)
        \right)
        =
        \sum_{\ell\in K\backslash L/L_g^H} 
        \cotr_{L_{\ell^{-1} g}^H}^L  
        \left(
        c_{\ell}^{\ast}       
        \tr_{K_{g}^{H}}^{L_g^H}  
        (m)
        \right).
    \]   
\end{proposition}
\begin{proof}
    Consider an element \(\cotr_{K_g^H}^{K}(m)\). We calculate the action of the transfer from \(K\) to \(L\)
    on spans.
    \[
    \begin{array}{rcl}
        \big(G/L \leftarrow G/L_{ag}^H \xrightarrow{=} G/L_{ag}^H \big)
        & \mapsto & 
        \big(G/L \leftarrow G/L_{ag}^H \xrightarrow{=} G/L_{ag}^H \big)
        \circ
        \big( G/K \leftarrow G/K \rightarrow G/L \big)   \\
        & = & 
        \big(G/K \leftarrow G/K\timesover{G/L} G/L_{ag}^H \rightarrow G/L_{ag}^H\big) \\
        & = & 
        G/K \leftarrow G\timesover{L} \big(L/K \times L/L_{ag}^H \big) \rightarrow  G/L_{ag}^H
    \end{array}
    \]    
    Using the double coset decomposition of Proposition \ref{prop:restrictionequation} and    
    noting that \(K \cap L_{lag}^H = K_{lag}^H \cap L\), the last term is equivalent to 
    \[
    \sum_{l \in K \backslash L / L_{ag}^H}
    \left(  G/K_{lag}^H \leftarrow G/(K \cap L_{lag}^H) \rightarrow G/L_{ag}^H   \right)
    \circ
    \big( G/K \leftarrow G/K_{lag}^H \xrightarrow{=} G/K_{lag}^H \big).
    \]
    We then combine this with \(\cotr_{K_g^H}^{K}(m)\) and only get a non-zero term when
    $l^{-1}=a$, which gives the formula. 
\end{proof}

\begin{corollary}
    If \(\mM\) is in \([H]^{\cO}\mhyphen\Mackey^G\)\index{Hmackey@\([H]^\cO\mhyphen\Mackey^G\)}, then for any \(L\geq [H]^{\cO}\)\index{LaboveH@\(L\geq [H]^{\cO}\)}, the restriction map
    \[
       \prod_{L_g^H\in T_{L\geq [H]^{\cO}}/L} \big(\mM(G/L_g^H)\big)^{W_L(L_g^H)}\cong \CoIndHO \mM(G/L)\to \CoIndHO\mM(G/L_g^{H})\cong\mM(G/L_g^H)
    \]
    is just the composite of the projection onto the factor indexed by \(L_g^H\) followed by the inclusion
    \[
        \big(\mM(G/L_g^H)\big)^{W_L(L_g^H)}\hookrightarrow \mM(G/L_g^H).
    \]
\end{corollary}

\begin{proposition}\label{prop: Identifying the Maps in the Unit}
    If \(\mM\) is in \(\AboveHMackey[{[H]^{\cO}}]^G\)\index{above@\(\AboveHMackey[{[H]^{\cO}}]^G\)}, then for all \(L\geq [H]^{\cO}\), the components of the unit map
    \[
        \mM(G/L)\to \CoIndHO\restrict{}{[H]^{\cO}} \mM(G/L)\cong\prod_{L_g^H\in T_{L\geq [H]^{\cO}}/L} \big(\mM(G/L_g^H)\big)^{W_L(L_g^H)}
    \]
    are given by the restriction maps:
    \[
        \res_{L_g^H}^{L}\colon \mM(G/L)\to \mM(G/L_g^H)^{W_L(L_g^H)}.
    \]
\end{proposition}

\subsection{Frobenius reciprocity}

The classical Frobenius relation shows that in the complete case, the left and right adjoints agree, rationally. The same is true for the incomplete case, see Theorem \ref{thm:left=rightAdjoint}. The proof of that theorem relies on reducing the problem
to the following lemma. 

\begin{lemma}\label{lem: Mackey Relation in Above H}
    If \(\mM\in\AboveHMackey[{[H]^{\cO}}]^G\)\index{above@\(\AboveHMackey[{[H]^{\cO}}]^G\)}, then for all \(L\geq [H]^{\cO}\), the composite
    \[
        \mM(G/L_g^H)\xrightarrow{tr_{L_g^H}^{L}} \mM(G/L)\xrightarrow{res_{L_{\tilde{g}}^H}^{L}} \mM(G/L_{\tilde{g}}^{H})
    \]
    is zero unless \(L_g^H\)\index{Lintersect@\(L_g^{H}\)} and \(L_{\tilde{g}}^{H}\) are \(L\)-conjugate. If they are \(L\)-conjugate by \(\ell'\), then the composite is
    \[
        \sum_{\ell\in N_{L}(L_g^H)/L_g^H} c_{\ell'}^\ast\circ c_{\ell}^\ast.
    \]
\end{lemma}
\begin{proof}
    We look at the double coset formula
    \[
        res_{L_{\tilde{g}}^H}^{L}\circ tr_{L_g^H}^{L}=\sum_{\ell \in L_g^H\backslash L/L_{\tilde{g}}^H} \tr_{L_g^H\cap \ell L_{\tilde{g}}^H\ell^{-1}}^{L_g^H}\circ c_{\ell}^\ast\circ \res_{\ell^{-1}L_g^H\ell\cap L_{\tilde{g}}^{H}}^{L_{\tilde{g}}^H}.
    \]
    If \(L_g^H\) and \(L_{\tilde{g}}^{H}\) are in distinct \(L\)-orbits, then by definition, 
    \[
        L_g^H\neq \ell L_{\tilde{g}}^H\ell^{-1},
    \]
    and hence 
    \[
        gHg^{-1}\neq \ell \tilde{g} H\tilde{g}^{-1}\ell^{-1}.
    \]
    By Corollary~\ref{cor:inseperability_representatives}, we therefore have that for all \(\ell\),
    \[
        L_g^H\cap \ell L_{\tilde{g}}^H\ell^{-1}\subseteq gHg^{-1}\cap \ell \tilde{g} H\tilde{g}^{-1}\ell^{-1}\not\in [H]^{\cO}.
    \]
    This means that the target of the restriction map is zero, which gives the first part.

    Now, if both are in the same \(L\)-orbit, then by conjugating by the corresponding value of \(\ell'\), we may without loss of generality assume that \(L_g^H=L_{\tilde{g}}^{H}\). If \(\ell\in L\) is not in the normalizer of \(gHg^{-1}\), then 
    \[
        L_g^H\cap \ell L_g^H {\ell}^{-1}\subseteq gHg^{-1}\cap \ell gHg^{-1}{\ell}^{-1}\not\in [H]^{\cO},
    \]
    and hence the target of the restriction for this summand is again zero. The only summands which can contribute in a non-zero way are therefore those for which 
    \[
        \ell\in N_{L}(gHg^{-1})=N_L(L_g^H),
    \]
    by Proposition~\ref{prop: Normalizer in L of LgH}. In this case, 
    \[
        L_g^H\cap \ell L_g^H\ell^{-1}=L_g^H,
    \]
    and we deduce the desired formula.
\end{proof}

Combining this with Proposition~\ref{prop: Indentifying the maps in the counit} and with Proposition~\ref{prop: Identifying the Maps in the Unit}, this gives us the desired Frobenius isomorphism in the rational case.

\begin{theorem}\label{thm:left=rightAdjoint}
    For any rational \(\mackabove_{[H]^{\cO}}\)-Mackey functor \(\mM\), the composite of the counit followed by the unit 
    \[
        \IndHO\restrict{}{[H]^{\cO}}\mM\to \mM \to \CoIndHO \restrict{}{[H]^{\cO}}\mM
    \]\index{Ind@\(\IndHO\)}\index{restrict@\(\restrict{}{[H]^{\cO}}\)}\index{coind@\(\CoIndHO\)}
    is an isomorphism.
\end{theorem}
\begin{proof}
Let \(\kappa\) denote the composite \(\eta_{\mM}\circ \epsilon_{\mM}\), and let \(L\) be a subgroup above \([H]^{\cO}\). Let
\[
    \iota_{L_g^H}\colon \big(\mM(G/L_g^H)\big)_{W_L(L_g^H)}\to \IndHO\mM(G/L)
\]
be the inclusion of the summand corresponding to \(L_g^H\), and dually let
\[
    \pi_{L_g^H}\colon \CoIndHO\mM(G/L)\to \big(\mM(G/L_g^H)\big)^{W_L(L_g^H)}
\]
be the projection onto the corresponding factor.

Consider the composite 
\[
\pi_{L_{\tilde{g}^H}}\circ\kappa\circ\iota_{L_g^H}\colon\big(\mM(G/L_g^H)\big)_{W_L(L_g^H)}\to
\big(\mM(G/L_{\tilde{g}}^H)\big)^{W_{L}(L_{\tilde{g}}^H)}.
\]
Proposition~\ref{prop: Indentifying the maps in the counit} identifies the first map in \(\kappa\) with sum of the transfers, while Proposition~\ref{prop: Identifying the Maps in the Unit} identifies the second with the product of the restrictions. The composite is therefore exactly the map considered in Lemma~\ref{lem: Mackey Relation in Above H}, and hence is zero unless \(L_g^H\) and \(L_{\tilde{g}}^H\) represent the same equivalence class. When they are the same equivalence class, the map is the usual trace map
\[
    \big(\mM(G/L_g^H)\big)_{W_L(L_g^H)}\to \big(\mM(G/L_{\tilde{g}}^H)\big)^{W_{L}(L_{\tilde{g}}^H)},
\]
which is a rational isomorphism.
\end{proof}

\begin{corollary}\label{cor: Identifying the Image of IndH}
    Let \(\mM\) be in \(\AboveHMackey[{[H]^{\cO}}]^G_\Q\)\index{above@\(\AboveHMackey[{[H]^{\cO}}]^G\)}. If for all \(L\geq [H]^{\cO}\), we have that the transfer maps
    \[
        \bigoplus_{L_g^H\in T_{L\geq [H]^{\cO}}/L} \mM(G/L_g^H)\to \mM(G/L)
    \]
    are surjective, then \(\mM\) is in the essential image of induction.
\end{corollary}
\begin{proof}
    Consider the counit of adjunction:
    \[
        \IndHO\restrict{}{[H]^{\cO}}\mM\to\mM.
    \]
    Theorem~\ref{thm:left=rightAdjoint} shows that this map is injective, and our assumption is that it is also surjective.
\end{proof}

\section{Understanding and simplifying the splitting}

\subsection{Understanding the splitting}

Corollary~\ref{cor: Idempotents are Above H} shows that if \(\mM\) is any \(e_{[H]^{\cO}}\mA^{\cO}_{\Q}\)-module, then \(\mM\) is in \(\AboveHMackey[{[H]^{\cO}}]^G_\Q\). In fact, it is automatically induced up from \([H]^{\cO}\mhyphen\Mackey^G_\Q\)\index{Hmackey@\([H]^\cO\mhyphen\Mackey^G_\Q\)}.

\begin{proposition}\label{prop:moduleisinduction}
    Any \(e_{[H]^{\cO}}\mA^{\cO}_{\Q}\)-module \(\mM\) is of the form \(\IndHO\restrict{}{[H]^{\cO}}\mM\).
\end{proposition}
\begin{proof}
    By Corollary~\ref{cor: Identifying the Image of IndH}, it suffices to show that the transfer maps
    \[
        \bigoplus_{L_g^H\in T_{L\geq [H]^{\cO}}/L}\mM(G/L_g^H)\to \mM(G/L)
    \]
    are surjective for any \(L\geq [H]^{\cO}\). By the usual Frobenius relation in a Mackey functor, it is sufficient to show that on \(\mM\),
    \[
        \res_L^G (e_{[H]^{\cO}})=\sum_{L_g^H\in T_{L\geq [H]^{\cO}}/L} \alpha_g tr_{L_g^H}^L(e_{L_g^H})
    \]
    for some virtual \(L_g^H\)-sets \(e_{L_g^H}\). We compute this using Lemma~\ref{lem:idempotentformula}:
    \[
        \res_L^G (e_{[H]^{\cO}})=\sum_{K\to H} \frac{|K|}{|N_GH|} \mu^{\cO}(K,H) i_L^\ast(G/K).
    \]
    Note that by construction, if \(K\to H\) and \(K\subsetneq H\), then \(K\) is not above \([H]^{\cO}\), and hence \(G/K\) in the Burnside ring acts as zero. This means that on \(\mM\), the idempotent \(e_{[H]^{\cO}}\) acts as:
    \[
        \frac{1}{W_{G}(H)} G/H.
    \]
    The result now follows from the double coset formula and the definition of \(T_{L\geq [H]^{\cO}}\)
\end{proof}

Since \(e_{[H]^{\cO}}\mA_{\Q}^{\cO}\) is an idempotent localization of \(\mA_{\Q}^{\cO}\), being an \(e_{[H]^{\cO}}\mA_{\Q}^{\cO}\)-module is a property of a rational \(\cO\)-Mackey functor, and the forgetful functor to rational \(\cO\)-Mackey functors is fully-faithful.

\begin{corollary}\label{cor: eH-Mod FF embeds}
    The forgetful functor is a fully-faithful functor
    \[
        e_{[H]^{\cO}}\mA_{\Q}^{\cO}\mhyphen\Mod\to [H]^{\cO}\mhyphen\Mackey^G_{\bQ}.
    \]
\end{corollary}

\begin{lemma}\label{lem:idempotentactionidentity}
    Let \(\mM \in [H]^{\cO}\mhyphen\Mackey^G_{\bQ}\)\index{Hmackey@\([H]^\cO\mhyphen\Mackey^G_\Q\)}. The idempotent
    \(e_{[H]^{\cO}}\) acts on \(\mM\) as the identity.
\end{lemma}
\begin{proof} 
    We calculate \(e_{[H]^{\cO}}\mM(G/L)\), which is only non-zero when \(L\) is above \([H]^{\cO}\).
    By Proposition \ref{prop:moduleisinduction}, the idempotent \(e_{[H]^{\cO}}\)\index{e@\(e_{[H]^{\cO}}\)} acts 
    on \(\mM \in [H]^{\cO}\mhyphen\Mackey^G\) as 
    \[
    \frac{1}{W_{G}(H)} G/H.
    \]
    The restriction of \(G/H\) to 
    an \(L\)-set is a coproduct of terms \(L/L \cap gHg^{-1}\)
    for \(g \in L \backslash G / H \). The action of 
    \(L/L \cap gKg^{-1}\) on \(\IndHO \mM (G/L)\) is given by
    \[
    \IndHO \mM (G/L)
    \xrightarrow{\res_{L \cap {g}Hg^{-1}}^{L}}
    \IndHO \mM (G/L \cap {g}Hg^{-1})
    \xrightarrow{\tr_{L \cap {g}Hg^{-1}}^{L}}
    \IndHO \mM (G/L)    
    \]
    and hence is zero whenever \(L \cap {g}Hg^{-1}\) is not above \([H]^{\cO}\).

    Assume that \(L \in [H]^{\cO}\). 
    Without loss of generality, we may assume \(L \subseteq H\).
    Consider a term \(L \cap {g}Hg^{-1}\). Then
    \(L \cap {g}Hg^{-1}\) is inside \(H \cap {g}Hg^{-1}\). We divide into two sub-cases. 
    Firstly, \(H \neq {g}Hg^{-1}\) so that \(H \cap {g}Hg^{-1}\) is a proper subgroup of \(H\) that is separated
    from \(H\) by the Meet Lemma (Lemma \ref{lem: Meet Lemma}).
    Thus, \(L \cap {g}Hg^{-1}\) is not above \([H]^{\cO}\) and the term contributes nothing. 
    Secondly, \(H = {g}Hg^{-1}\) and \(L \cap {g}Hg^{-1}=L\). This gives \(\id=L/L\)
    acting on \(\IndHO \mM (G/L) \). The second sub-case occurs \(|N_G H|/|H|\) many times,
    hence \(e_{[H]^{\cO}}\) acts as the identity on \(\IndHO \mM (G/L)\).
    
    Now assume that \(L\) is above  \([H]^{\cO}\) but not inside it. 
    Corollary \ref{cor:transfers_generate_value_Ind}
    states that any element of 
    \(\IndHO \mM (G/L)\) is a sum of terms of the form \(\tr_J^L m\)
    for \(m \in \mM (G/J)\) and \(J \in [H]^{\cO}\). Using Frobenius reciprocity
    and the first case 
    \[
        (\res_L^G e_{[H]^{\cO}})( \tr_J^L (m)) = \tr_J^L ( \res_J^G e_{[H]^{\cO}}( m )) = \tr_J^L (m). \qedhere
    \]
\end{proof}

\begin{corollary}\label{cor: All IndH Modules Hit}
    The forgetful functor
    \[
        e_{[H]^{\cO}}\mA_{\Q}^{\cO}\mhyphen\Mod\to [H]^{\cO}\mhyphen\Mackey^G_{\bQ}
    \]
    is essentially surjective.
\end{corollary}

\begin{theorem}\label{thm:splitting2} 
There is an equivalence of categories
     \[
        \OMackey_{\Q}^G\cong \prod_{[H]^{\cO}} [H]^\cO\mhyphen\Mackey^G_{\Q}.
    \]\index{omackey@\(\OMackey^G_\Q\)}\index{Hmackey@\([H]^\cO\mhyphen\Mackey^G_\Q\)}
\end{theorem}
\begin{proof} 
    Corollary~\ref{cor:splitting O Mackey} shows that we have an equivalence of categories
    \[
        \OMackey_{\Q}^G\simeq\prod_{[H]^{\cO}} e_{[H]^{\cO}}\mA_{\Q}^{\cO}\mhyphen\Mod.
    \]
    Corollaries~\ref{cor: eH-Mod FF embeds} and \ref{cor: All IndH Modules Hit} complete the proof.
\end{proof}

\subsection{Simplifying the splitting}
Greenlees and May's splitting has an additional simplification: there is no need to remember conjugate subgroups, since the Mackey functor will have isomorphic values on orbits for conjugate subgroups. Thus for a maximal transfer system, \(e_{(H)}\mM\) is generated from a single value \(e_{(H)}\mM(G/H)\) together with the Weyl group action on it. In this subsection we describe a similar simplification for a general transfer system \(\cO\).

\begin{definition}
    Let \(\Orb^{G}_{[H]^{\cO}}\)\index{orb@\(\Orb^{G}_{[H]^{\cO}}\)} be the full subcategory of the orbit category of \(G\) spanned by those orbits \(G/J\) with \(J\in [H]^{\cO}\).
\end{definition}

When \(\cO\) is the terminal transfer system, then the only subgroups inseparable from \(H\) are the conjugates of \(H\), and so this is the full subcategory of the orbit category spanned by those orbits. Since every map \(G/H\to G/H'\) with \(H\) and \(H'\) conjugate is an isomorphism, this is a groupoid, and it is equivalent to the group \(W_G(H)\). We have a similar story here.

\begin{definition}
    Let \(H \to G \in \cO\) and \(N=N_G(H)\). Let 
    \[
        \Orb^N_{\langle H\rangle^{\cO}}=\langle N/J\mid J\in [H]^{\cO}\text{ and }J\subseteq H\rangle
    \]\index{orb@\(\Orb^N_{\langle H\rangle^{\cO}}\)}
    be the full subcategory of the orbit category of \(N\) spanned by those orbits \(N/J\) with \(J\in\Sub_{\langle H\rangle}^{\cO}\)\index{sub@\(\Sub_{\langle H\rangle}^{\cO}\)}.
\end{definition}

\begin{proposition}\label{prop: tombstone N simplification}
    Let \(H \to G \in \cO\) and \(N=N_G(H)\). Induction gives an equivalence of categories
    \[
        G\timesover{N}(\mhyphen)\colon \Orb^{N}_{\langle H\rangle^{\cO}}\to\Orb^{G}_{[H]^{\cO}}.
    \]
\end{proposition}
\begin{proof}
    Proposition~\ref{prop: tombstones} shows that induction is essentially surjective. Induction is always faithful, so we need only show fullness. For this, we have to compute, for \(J,K\in [H]^{\cO}\) and \(J,K\subseteq H\), the set of maps
    \[
        \Map^G(G/K,G/J)\cong \{gJ\mid K\subset gJg^{-1}\}.
    \]
    Now since \(J\subseteq H\), we have \(gJg^{-1}\subseteq gHg^{-1}\). If \(K\subseteq gJg^{-1}\), then it is in both \(H\) and \(gHg^{-1}\), so being inseparable from \(H\) implies that \(H=gHg^{-1}\), and hence \(g\in N\), by Proposition~\ref{prop: J not in any distinct conjugate of H}.
\end{proof}

\begin{remark}
    Unlike in the complete case, we cannot in general descend down to Weyl-equivariance. This is because there is no reason to think that some element \(h\in H\) normalizes all of the subgroups \(K\) inseparable from \(H\). The example of the initial transfer systems, where \(Orb^{G}_{\langle G\rangle^{\cO}}=\Orb^G\), for \(G\) non-abelian shows quite nicely how badly things can fail.
\end{remark}

If there are no transfers connecting subgroups in some particular inseparability class \([H]^\cO\), then this completely defines the category of \([H]^\cO\)-Mackey functors, see Remark \ref{rem:structure_of_(H)-Mackey}. 

\begin{definition}
    Let \(\Coeff_{\langle H\rangle^{\cO}}^{N}\)\index{coeff@\(\Coeff_{\langle H\rangle^{\cO}}^{N}\)} denote the category of contravariant functors
    \[
        \Orb^{N}_{\langle H\rangle^{\cO}}\to\Ab.
    \]
\end{definition}

\begin{corollary}\label{cor: simplified-[H]Mackey for disklike}
    Suppose \(\cO\) is a transfer system, such that there are no transfers in \([H]^\cO\)\index{H@\([H]^\cO\)} and let \(N=N_G(H)\). Then we have an equivalence of categories
    \[
        [H]^\cO\mhyphen\Mackey^G\simeq \Coeff_{\langle H\rangle^{\cO}}^{N}.
    \]\index{Hmackey@\([H]^\cO\mhyphen\Mackey^G\)}
\end{corollary}

In other words, it is enough to remember the values of the \([H]^{\cO}\)-Mackey functor only on orbits \(G/K\) where \(K\subseteq H\) and \(K\in [H]^{\cO}\) together with the restrictions and \(N_G(H)\) action on this diagram. The other values of an \([H]^{\cO}\)-Mackey functor can be recovered from these using conjugation by elements of \(G\). 

The result above applies in particular to all inseparability classes when \(\cO\) is a disk-like transfer system. When we work rationally, combining Theorem \ref{thm:splitting2} and Corollary \ref{cor: simplified-[H]Mackey for disklike}, we obtain the following.

\begin{theorem}\label{thm:splitting2simplified_disklike} Let \(\cO\) be a disk-like  transfer system. There is an equivalence of categories
     \[
        \OMackey^G_\Q\cong \prod_{[H]^{\cO}} 
        \Coeff_{\langle H\rangle^{\cO},{\Q}}^{N_G(H)}.
    \]\index{omackey@\(\OMackey^G_\Q\)}\index{coeff@\(\Coeff_{\langle H\rangle^{\cO}}^{N}\)}
\end{theorem}

Notice that for the maximal transfer system, which is disk-like, Theorem  \ref{thm:splitting2simplified_disklike} gives precisely the Greenlees--May splitting.

When there are internal transfers in \([H]^\cO\), we can still connect \([H]^\cO\)-Mackey functors to \(N\)-Mackey functors, where \(N=N_G(H)\). Consider the full subcategory of the \(i^*_N\cO\)-Lindner category for \(N\) spanned by those objects \(N/K\) with \(K\subseteq H\): \(\cA^{N,\cO}_{(H)}\). Product preserving functors out of this are \(i^*_N\cO\)-Mackey functors for \(N\) on the family of \(H\) and its subgroups (since \(H\) is normal in \(N\)). 

\begin{definition}
    Let \(\langle H\rangle^{\cO}\mhyphen\Mackey^N\)\index{Hmackey@\(\langle H\rangle^{\cO}\mhyphen\Mackey^N\)} denote the full subcategory of \(i^*_N\cO\)-Mackey functors \(\mM\), for \(N\) on \(H\) and its subgroups, with the property that \(\mM(N/J)=0\) if \(J\not\in [H]^{\cO}\).
\end{definition}

Restriction from \(\cO\)-Mackey functors for \(G\) to \(i^*_N\cO\)-Mackey functors for \(N\) then gives the following.
\begin{proposition}\label{prop: <H>MackeyN}
    Let \(H\to G \in \cO\). Restriction gives an equivalence of categories
    \[
        [H]^\cO\mhyphen\Mackey^G\to\langle H\rangle^{\cO}\mhyphen\Mackey^{N}.
    \]
\end{proposition}

\begin{proof}
    This follows from the equivalence of categories in Proposition \ref{prop: tombstone N simplification}.
\end{proof}

When we work rationally, combining Theorem \ref{thm:splitting2} and Proposition \ref{prop: <H>MackeyN}, we obtain the following.

\begin{theorem}\label{thm: simplified-noconjugation}
    Let \(\cO\) be a transfer system for \(G\). There is an equivalence of categories
     \[
        \OMackey^G_\Q\cong \prod_{[H]^{\cO}} \langle H\rangle^{\cO}\mhyphen\Mackey^{N_G(H)}_{\Q}.
    \]\index{Hmackey@\(\langle H\rangle^{\cO}\mhyphen\Mackey^N\)}
\end{theorem}

\section{Examples}\label{sec:examples}

In this section we illustrate our results on several simple examples. We begin with the splitting for the incomplete Burnside ring Mackey functors and corresponding splittings of general incomplete Mackey functors. Recall the incomplete Burnside ring from Example \ref{ex:c_6_Burnside}.

\begin{ex}
 Suppose \(G=C_6\) and the transfer system is given by the identity maps, \(C_1 \to C_3\) and \(C_2 \to C_6\).
 The rational \(C_6\)-Burnside ring for \(\cO\) is the following incomplete \(\cO\)-Mackey functor, where in the brackets we indicate the additive basis for each \(\bQ\)-module.

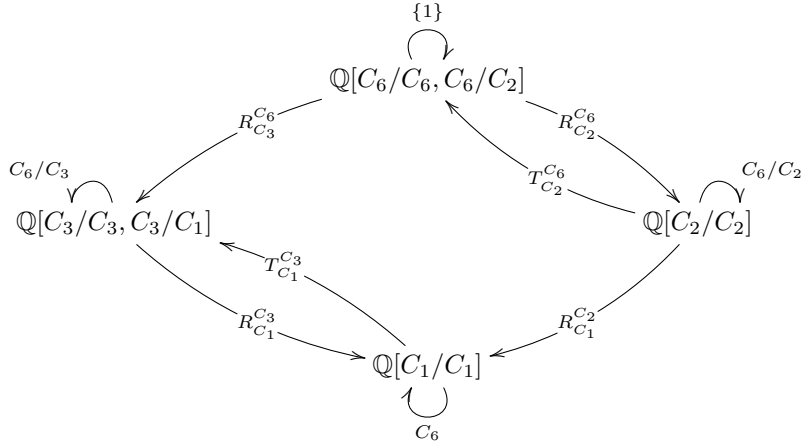
\[
\xymatrix@R+0.5cm@C+0.5cm{
& \bQ[C_6/C_6, C_6/C_2]
\ar@(ul,ur)^{\trivgp}
\ar@/^1.pc/[dr]|{R_{C_2}^{C_6}}
\ar@/_1.pc/[dl]|{R_{C_3}^{C_6}}
\\
\bQ[C_3/C_3,C_3/C_1]
\ar@(u,l)_{C_6/C_3}
\ar@/_1.pc/[dr]|{R_{C_1}^{C_3}}
&&
\bQ[C_2/C_2]
\ar@(u,r)^{C_6/C_2}
\ar@/^1.pc/[dl]|{R_{C_1}^{C_2}}
\ar@/^1.pc/[ul]|{T_{C_2}^{C_6}}
\\
&
\bQ[C_1/C_1]
\ar@(dr,dl)^{C_6}
\ar@/_1.pc/[ul]|{T_{C_1}^{C_3}}
}
\]

This incomplete Mackey functor splits into a direct sum of two diagrams of the following form:

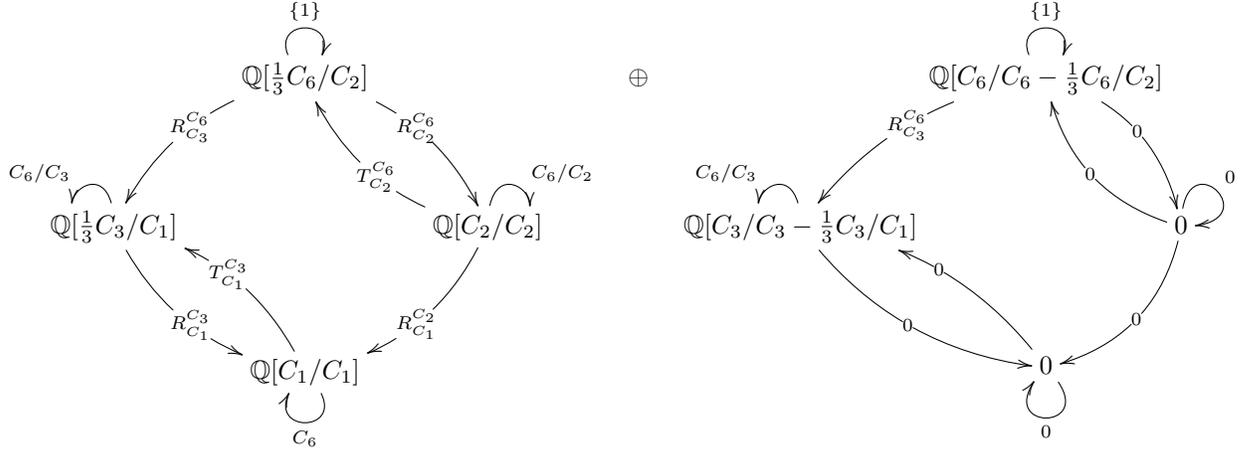
\[
\xymatrix@R+0.5cm@C-0.2cm{
& \bQ[\frac{1}{3}C_6/C_2]
\ar@(ul,ur)^{\trivgp}
\ar@/^1.pc/[dr]|{R_{C_2}^{C_6}}
\ar@/_1.pc/[dl]|{R_{C_3}^{C_6}}
\\
\bQ[\frac{1}{3}C_3/C_1]
\ar@(u,l)_{C_6/C_3}
\ar@/_1.pc/[dr]|{R_{C_1}^{C_3}}
&&
\bQ[C_2/C_2]
\ar@(u,r)^{C_6/C_2}
\ar@/^1.pc/[dl]|{R_{C_1}^{C_2}}
\ar@/^1.pc/[ul]|{T_{C_2}^{C_6}}
\\
&
\bQ[C_1/C_1]
\ar@(dr,dl)^{C_6}
\ar@/_1.pc/[ul]|{T_{C_1}^{C_3}}
}
{\bigoplus}
\xymatrix@R+0.5cm@C-0.9cm{
& \bQ[C_6/C_6-\frac{1}{3}C_6/C_2]
\ar@(ul,ur)^{\trivgp}
\ar@/_1.pc/[dl]|{R_{C_3}^{C_6}}
\ar@/^1.pc/[dr]|{0}
\\
\bQ[C_3/C_3-\frac{1}{3}C_3/C_1]
\ar@(u,l)_{C_6/C_3}
\ar@/_1.pc/[dr]|{0}
&&
0
\ar@(u,r)^{0}
\ar@/^1.pc/[dl]|{0}
\ar@/^1.pc/[ul]|{0}
\\
&
0
\ar@(dr,dl)^{0}
\ar@/_1.pc/[ul]|{0}
}
\]

In both diagrams all named maps are isomorphisms. 
The left diagram corresponds to idempotent \(e_{[2]}=e_1+e_2\) and the right one to the idempotent \(e_{[6]}=e_3+e_6\). We used the shortened notation \(e_n\) for \(e_{C_n}\) and we expressed the idempotents on the left in terms of idempotents of the complete rational Burnside ring for \(C_6\), see 
Corollary \ref{cor:idempotentformula}.

Continuing this example, a general  \(\cO\)-Mackey functor \(M\) for \(C_6\) is of the following form.

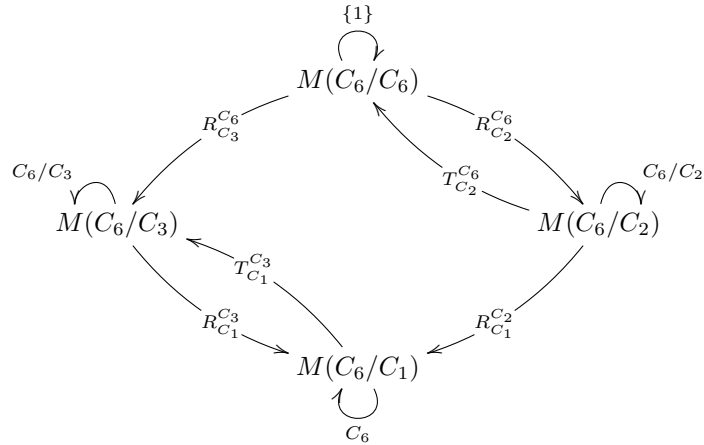
\[
\xymatrix@R+0.5cm@C+0.5cm{
& M(C_6/C_6)
\ar@(ul,ur)^{\trivgp}
\ar@/^1.pc/[dr]|{R_{C_2}^{C_6}}
\ar@/_1.pc/[dl]|{R_{C_3}^{C_6}}
\\
M(C_6/C_3)
\ar@(u,l)_{C_6/C_3}
\ar@/_1.pc/[dr]|{R_{C_1}^{C_3}}
&&
M(C_6/C_2)
\ar@(u,r)^{C_6/C_2}
\ar@/^1.pc/[dl]|{R_{C_1}^{C_2}}
\ar@/^1.pc/[ul]|{T_{C_2}^{C_6}}
\\
&
M(C_6/C_1)
\ar@(dr,dl)^{C_6}
\ar@/_1.pc/[ul]|{T_{C_1}^{C_3}}
}
\]

The splitting of a general \(\cO\)-Mackey functor \(M\) for \(C_6\) is analogous to the one of the \(C_6\)-Burnside ring for \(\cO\), however the coloured maps below do not have to be isomorphisms. 

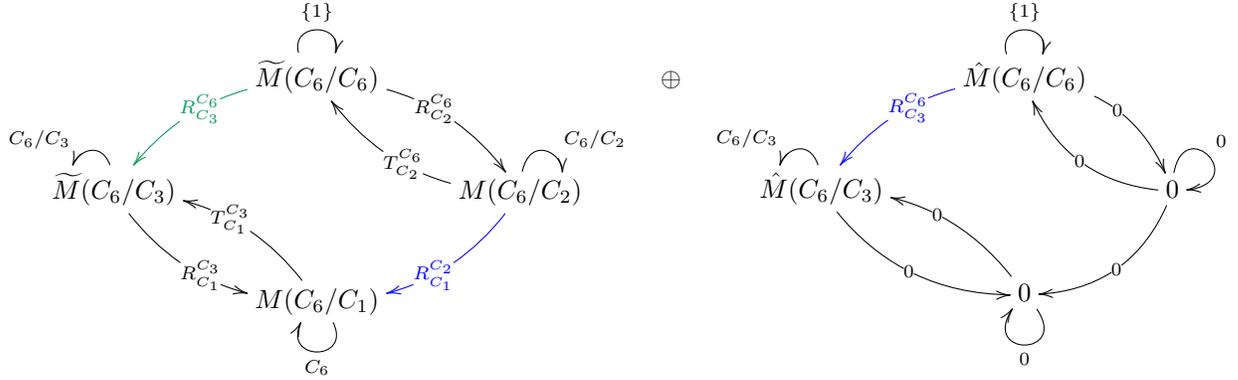
\[
\xymatrix@R+0.cm@C+0.cm{
& \widetilde{M}(C_6/C_6)
\ar@(ul,ur)^{\trivgp}
\ar@/^1.pc/[dr]|{R_{C_2}^{C_6}}
\ar@[ForestGreen]@/_1.pc/[dl]|{\green{R_{C_3}^{C_6}}}
\\
\widetilde{M}(C_6/C_3)
\ar@(u,l)_{C_6/C_3}
\ar@/_1.pc/[dr]|{R_{C_1}^{C_3}}
&&
M(C_6/C_2)
\ar@(u,r)^{C_6/C_2}
\ar@[blue]@/^1.pc/[dl]|{\blue{R_{C_1}^{C_2}}}
\ar@/^1.pc/[ul]|{T_{C_2}^{C_6}}
\\
&
M(C_6/C_1)
\ar@(dr,dl)^{C_6}
\ar@/_1.pc/[ul]|{T_{C_1}^{C_3}}
}
{\bigoplus}
\xymatrix@R+0cm@C+0cm{
& \hat{M}(C_6/C_6)
\ar@(ul,ur)^{\trivgp}
\ar@[blue]@/_1.pc/[dl]|{\blue{R_{C_3}^{C_6}}}
\ar@/^1.pc/[dr]|{0}
\\
\hat{M}(C_6/C_3)
\ar@(u,l)_{C_6/C_3}
\ar@/_1.pc/[dr]|{0}
&&
0
\ar@(u,r)^{0}
\ar@/^1.pc/[dl]|{0}
\ar@/^1.pc/[ul]|{0}
\\
&
0
\ar@(dr,dl)^{0}
\ar@/_1.pc/[ul]|{0}
}
\]

Here \(M(C_6/C_6)= \widetilde{M}(C_6/C_6) \oplus \hat{M}(C_6/C_6)\) and similarly 
\(M(C_6/C_3)= \widetilde{M}(C_6/C_3) \oplus \hat{M}(C_6/C_3)\).

Notice that the green map is completely determined by the blue one in the same diagram. Therefore after splitting, the data we have to remember to completely recover an \(\cO\)-Mackey functor \(M\) for \(C_6\) is illustrated in the following diagram.

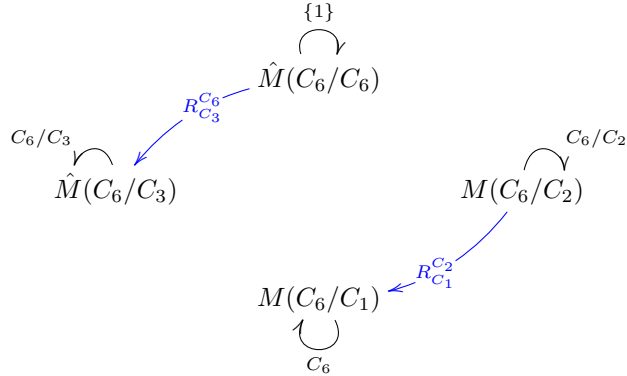
\[
\xymatrix@R+0cm@C+0cm{
& \hat{M}(C_6/C_6)
\ar@(ul,ur)^{\trivgp}
\ar@[blue]@/_1.pc/[dl]|{\blue{R_{C_3}^{C_6}}}
\\
\hat{M}(C_6/C_3)
\ar@(u,l)_{C_6/C_3}
&& 
M(C_6/C_2)
\ar@(u,r)^{C_6/C_2}
\ar@[blue]@/^1.pc/[dl]|{\blue{R_{C_1}^{C_2}}}
\\
&
M(C_6/C_1)
\ar@(dr,dl)^{C_6}
}
\]
\end{ex}

We consider another example of the splitting of the incomplete Burnside ring and the corresponding splitting of the category of incomplete Mackey functors.

\begin{ex}
    Let \(G=C_8\) and take the transfer system consisting of identity maps, \(C_2 \to C_4\) and \(C_2 \to C_8\).
    The rational \(C_8\)-Burnside ring for \(\cO\) is the following incomplete Mackey functor, which splits. In the bracket we indicate the additive basis for each \(\bQ\)-module.

\[
\xymatrix@R+0.5cm@C+0.5cm{
\bQ[C_8/C_8,C_8/C_2]
\ar@(ul,ur)^{\trivgp}
\ar@/_0.pc/[d]|{R_{C_4}^{C_8}}
\\
\bQ[C_4/C_4,C_4/C_2]
\ar@(u,r)^{C_8/C_4}
\ar@/_0.pc/[d]|{R_{C_2}^{C_4}}
\\
\bQ[C_2/C_2]
\ar@(u,r)^{C_8/C_2}
\ar@/_0.pc/[d]|{R_{C_1}^{C_2}}
\ar@/^1.pc/[u]^{T_{C_2}^{C_4}}
\ar@/^3.5pc/[uu]^{T_{C_2}^{C_8}}
\\
\bQ[C_1/C_1]
\ar@(dr,dl)^{C_8}
}
\quad = \quad
\xymatrix@R+0.5cm@C+0.5cm{
\bQ[\frac{1}{4}C_8/C_2]
\ar@(ul,ur)^{\trivgp}
\ar@/_0.pc/[d]|{R_{C_4}^{C_8}}
\\
\bQ[\frac{1}{2}C_4/C_2]
\ar@(u,r)^{C_8/C_4}
\ar@/_0.pc/[d]|{R_{C_2}^{C_4}}
\\
\bQ[C_2/C_2]
\ar@(u,r)^{C_8/C_2}
\ar@/_0.pc/[d]|{R_{C_1}^{C_2}}
\ar@/^1.pc/[u]^{T_{C_2}^{C_4}}
\ar@/^3.5pc/[uu]^{T_{C_2}^{C_8}}
\\
\bQ[C_1/C_1]
\ar@(dr,dl)^{C_8}
}
\quad \oplus \quad
\xymatrix@R+0.5cm@C+0.5cm{
\bQ[C_8/C_8-\frac{1}{4}C_8/C_2]
\ar@(ul,ur)^{\trivgp}
\ar@/_0.pc/[d]|{R_{C_4}^{C_8}}
\\
\bQ[C_4/C_4-\frac{1}{2}C_4/C_2]
\ar@(u,r)^{C_8/C_4}
\ar@/_0.pc/[d]|{0}
\\
0
\ar@/_0.pc/[d]|{0}
\\
0
}
\]

This corresponds to the idempotent splitting \(1=e_{[2]}+ e_{[8]}\), where \(e_{[2]}=e_1+e_2,  e_{[8]}=e_4 +e_8\) in terms of idempotents of the complete rational Burnside ring for \(C_8\). 

The splitting of a general \(\cO\)-Mackey functor for \(C_8\) is analogous. Below we indicate in blue the data that determines the whole left hand side. 

\[
\xymatrix@R+0.5cm@C+0.5cm{
M(C_8/C_8)
\ar@(ul,ur)^{\trivgp}
\ar@/_0.pc/[d]|{R_{C_4}^{C_8}}
\\
M(C_8/C_4)
\ar@(u,r)^{C_8/C_4}
\ar@/_0.pc/[d]|{R_{C_2}^{C_4}}
\\
M(C_8/C_2)
\ar@(u,r)^{C_8/C_2}
\ar@/_0.pc/[d]|{R_{C_1}^{C_2}}
\ar@/^1.pc/[u]^{T_{C_2}^{C_4}}
\ar@/^3.5pc/[uu]^{T_{C_2}^{C_8}}
\\
M(C_8/C_1)
\ar@(dr,dl)^{C_8}
}
\quad = \quad
\xymatrix@R+0.5cm@C+0.5cm{
\widetilde{M}(C_8/C_8)
\ar@(ul,ur)^{\trivgp}
\ar@/_0.pc/[d]|{R_{C_4}^{C_8}}
\\
\widetilde{M}(C_8/C_4)
\ar@(u,r)^{C_8/C_4}
\ar@/_0.pc/[d]|{R_{C_2}^{C_4}}
\\
\blue{M(C_8/C_2)}
\ar@[blue]@(u,r)^{\blue{C_8/C_2}}
\ar@[blue]@/_0.pc/[d]|{\blue{R_{C_1}^{C_2}}}
\ar@/^1.pc/[u]^{T_{C_2}^{C_4}}
\ar@/^3.5pc/[uu]^{T_{C_2}^{C_8}}
\\
\blue{M(C_8/C_1)}
\ar@[blue]@(dr,dl)^{\blue{C_8}}
}
\quad \oplus \quad
\xymatrix@R+0.5cm@C+0.5cm{
\blue{\hat{M}(C_8/C_8)}
\ar@[blue]@(ul,ur)^{\blue{\trivgp}}
\ar@[blue]@/_0.pc/[d]|{\blue{R_{C_4}^{C_8}}}
\\
\blue{\hat{M}(C_8/C_4)}
\ar@[blue]@(u,r)^{\blue{C_8/C_4}}
\ar@/_0.pc/[d]|{0}
\\
0
\ar@/_0.pc/[d]|{0}
\\
0
}
\]

\end{ex}

\begin{ex}
   Take the transfer system for \(C_8\) consisting of the identity maps and \(C_2 \to C_4\).
    This gives the following incomplete Burnside ring Mackey functor, which does not split.

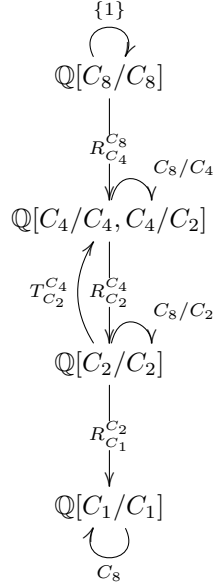
\[
\xymatrix@R+0.5cm@C+0.5cm{
\bQ[C_8/C_8]
\ar@(ul,ur)^{\trivgp}
\ar@/_0.pc/[d]|{R_{C_4}^{C_8}}
\\
\bQ[C_4/C_4,C_4/C_2]
\ar@(u,r)^{C_8/C_4}
\ar@/_0.pc/[d]|{R_{C_2}^{C_4}}
\\
\bQ[C_2/C_2]
\ar@(u,r)^{C_8/C_2}
\ar@/_0.pc/[d]|{R_{C_1}^{C_2}}
\ar@/^1.pc/[u]^{T_{C_2}^{C_4}}
\\
\bQ[C_1/C_1]
\ar@(dr,dl)^{C_8}
}
\]

The maximal disk-like transfer system in the poset of transfer systems for \(C_8\) that is smaller than the transfer system of this example is the trivial one. Therefore there is no splitting of this incomplete Burnside ring Mackey functor, nor of this category of incomplete Mackey functors.
\end{ex}

In the next four examples we concentrate on the case \(G=C_6\). We will illustrate all the possible structures of incomplete Mackey functors and rationally which idempotents exist.  
We begin by looking at the disk-like transfer systems. 

\begin{ex}\label{ex:list_for_C6}
Let \(G=C_6\). 
We will use the abbreviation \(e_n\) to denote idempotent \(e_{C_n}\) in the \emph{complete} Burnside ring for \(C_6\). 
In every case the diagram illustrates the incomplete Mackey functor structure for the chosen transfer system, with only additive transfers drawn, as restriction and conjugation maps are always present. We abbreviate the values by \(\bullet\), but our convention is as in the previous examples for \(G=C_6\):
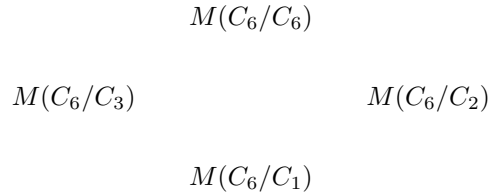
\[
\xymatrix@R0.5cm@C0.5cm{
& M(C_6/C_6)
\\
M(C_6/C_3)
&& 
M(C_6/C_2)
\\
&
M(C_6/C_1)
}
\]

\begin{enumerate}
\item Let \(\cO\) be the trivial transfer system, then there is no idempotent splitting. 
The \(\cO\)-admissible orbits are the trivial ones.

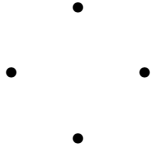
\[
\xymatrix@R0.5cm@C0.5cm{
& \bullet
\\
\bullet
&& 
\bullet
\\
&
\bullet
}
\]

\item Let \(\cO\) be the transfer system generated by \(C_3 \to C_6\) then the maximal splitting into idempotents is \(e_{[3]}=e_1+e_3, e_{[6]}=e_2+e_6\).
The \(\cO\)-admissible orbits are the trivial ones, \(C_6/C_3\) and \(C_2/C_1\).

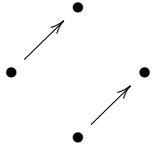
\[
\xymatrix@R0.5cm@C0.5cm{
& \bullet
\\
\bullet
\ar[ur]
&& 
\bullet
\\
&
\bullet
\ar[ur]
}
\]

\item Let \(\cO\) be the transfer system generated by \(C_2 \to C_6\) then the maximal splitting into idempotents is \(e_{[2]}=e_1+e_2, e_{[6]}=e_3+e_6\).
The \(\cO\)-admissible orbits are the trivial ones, \(C_6/C_2\) and \(C_3/C_1\).

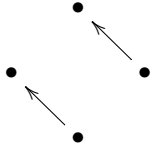
\[
\xymatrix@R0.5cm@C0.5cm{
& \bullet
\\
\bullet
&& 
\bullet
\ar[ul]
\\
&
\bullet
\ar[ul]
}
\]

\item Let \(\cO\) be the transfer system generated by \(C_1 \to C_6\) then the maximal splitting into idempotents is \(e_1, e_{[6]}=e_2+e_3+e_6\).
The \(\cO\)-admissible orbits are the trivial ones, \(C_6/C_1\), \(C_3/C_1\) and \(C_2/C_1\).

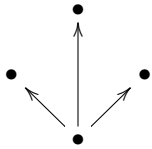
\[
\xymatrix@R0.5cm@C0.5cm{
& \bullet
\\
\bullet
&& 
\bullet
\\
&
\bullet
\ar[ul]
\ar[ur]
\ar[uu]
}
\]

\item Let \(\cO\) be the transfer system generated by \(C_2 \to C_6\) and  \(C_3 \to C_6 \) then the maximal splitting into idempotents is \(e_1, e_2, e_3, e_6\).
All orbits are \(\cO\)-admissible.

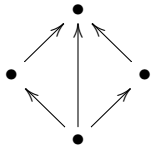
\[
\xymatrix@R0.5cm@C0.5cm{
& \bullet
\\
\bullet
\ar[ur]
&& 
\bullet
\ar[ul]
\\
&
\bullet
\ar[ul]
\ar[ur]
\ar[uu]
}
\]

\item Let \(\cO\) be the transfer system generated by \(C_1 \to C_6\) and  \(C_3 \to C_6 \) then the maximal splitting into idempotents is \(e_1, e_3, e_{[6]}= e_2+e_6\).
The \(\cO\)-admissible orbits are the trivial ones, \(C_6/C_3\), \(C_6/C_1\), \(C_3/C_1\) and \(C_2/C_1\).

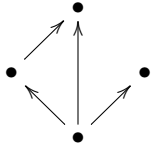
\[
\xymatrix@R0.5cm@C0.5cm{
& \bullet
\\
\bullet
\ar[ur]
&& 
\bullet
\\
&
\bullet
\ar[ul]
\ar[ur]
\ar[uu]
}
\]

\item Let \(\cO\) be the transfer system generated by \(C_1 \to C_6\) and  \(C_2 \to C_6 \) then the maximal splitting into idempotents is \(e_1, e_2, e_{[6]}=e_3+e_6\).
Admissible orbits are the trivial ones, \(C_6/C_2\), \(C_6/C_1\), \(C_3/C_1\) and \(C_2/C_1\).

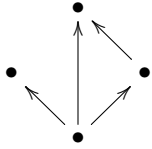
\[
\xymatrix@R0.5cm@C0.5cm{
& \bullet
\\
\bullet
&& 
\bullet
\ar[ul]
\\
&
\bullet
\ar[ul]
\ar[ur]
\ar[uu]
}
\]

\end{enumerate}
\end{ex}

\begin{ex}
    There are three more transfer systems for \(C_6\) that are not disk-like. 
    These cases have no splittings. 

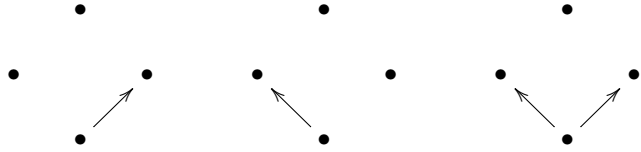
\[
\xymatrix@R0.5cm@C0.5cm{
& \bullet
\\
\bullet
&& 
\bullet
\\
&
\bullet
\ar[ur]
}
\quad  \quad \quad
\xymatrix@R0.5cm@C0.5cm{
& \bullet
\\
\bullet
&& 
\bullet
\\
&
\bullet
\ar[ul]
}
\quad \quad \quad
\xymatrix@R0.5cm@C0.5cm{
& \bullet
\\
\bullet
&& 
\bullet
\\
&
\bullet
\ar[ur]
\ar[ul]
}
\]

\end{ex}

Notice that we never have a splitting given by \(e_1+e_6, e_2, e_3\), nor \(e_1+e_2+e_3,e_6\). Such splittings cannot happen, see Corollary \ref{cor: Relative Family} for the first case and Corollary \ref{cor:inseperability_representatives} for the second one.

We now proceed to analysing the splitting in two cases from Example \ref{ex:list_for_C6}.

\begin{ex}
    For number (4) in Example \ref{ex:list_for_C6}, we illustrate the splitting of a general Mackey functor. The structure of a general Mackey functor that we need to remember after splitting is indicated in  green below. We put restriction maps in these diagrams (going down) only if they need to be remembered. Similarly, the conjugation action is not marked explicitly, but it needs to be remembered on the  green values. 

\[
\xymatrix@R0.5cm@C0.5cm{
& \bullet
\\
\bullet
&& 
\bullet
\\
&
\bullet
\ar[ul]
\ar[ur]
\ar[uu]
}
\quad = \quad 
\xymatrix@R0.5cm@C0.5cm{
& \bullet
\\
\bullet
&& 
\bullet
\\
&
\green{\bullet}
\ar[ul]
\ar[ur]
\ar[uu]
} 
\quad \oplus \quad
\xymatrix@R0.5cm@C0.5cm{
&\green{ \bullet}
\ar@[ForestGreen][dr]
\ar@[ForestGreen][dl]
\\
\green{\bullet}
&& 
\green{\bullet}
\\
&
0
}
\]
\end{ex}

\begin{ex}
    For number (6) in Example \ref{ex:list_for_C6}, we illustrate the splitting. The structure of a general Mackey functor that we need to remember after splitting is indicated in green below. We put restriction maps in these diagrams (going down) only if they need to be remembered. Similarly, the conjugation action is not marked explicitly, but it needs to be remembered on the green values. 

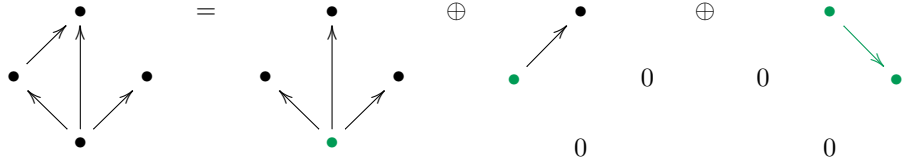
\[
\xymatrix@R0.5cm@C0.5cm{
& \bullet
\\
\bullet
\ar[ur]
&& 
\bullet
\\
&
\bullet
\ar[ul]
\ar[ur]
\ar[uu]
}
\quad = \quad 
\xymatrix@R0.5cm@C0.5cm{
& \bullet
\\
\bullet
&& 
\bullet
\\
&
\green{\bullet}
\ar[ul]
\ar[ur]
\ar[uu]
}
\quad \oplus \quad
\xymatrix@R0.5cm@C0.5cm{
& \bullet
\\
\green{\bullet}
\ar[ur]
&& 
0
\\
&
0
}
\quad \oplus \quad
\xymatrix@R0.5cm@C0.5cm{
& \green{\bullet}
\ar@[ForestGreen][dr]
\\
0
&& 
\green{\bullet}
\\
&
0
}
\]
\end{ex}

We now illustrate the difference between the split pieces when we use a disk-like transfer system and a non-disk-like one. The statement of the result is the same, but in both cases \([H]^{\cO}\mhyphen\Mackey^G_\Q\) denote different categories, depending on the transfer systems. 

\begin{ex}\label{ex:difference_disklike-non-disklike}

Let \(G=C_8\) and \(\cO\) be the following transfer system, as in Example \ref{ex:transfers_withinClasses}.

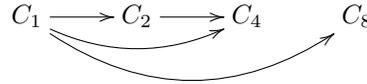
\[
\xymatrix@R+0cm@C+0cm{
C_1 
\ar@/^0.pc/[r]
\ar@/_1.pc/[rr]
\ar@/_2.pc/[rrr]
&
C_2
\ar@/^0.pc/[r]
&
C_4 
&
C_8
}
\]

\vspace{1cm}

Again, the maximal disk-like transfer system \(\cO^d \subseteq \cO\) is the following.

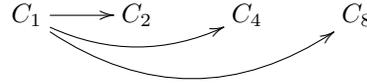
\[
\xymatrix@R+0cm@C+0cm{
C_1 
\ar@/^0.pc/[r]
\ar@/_1.pc/[rr]
\ar@/_2.pc/[rrr]
&
C_2
&
C_4 
&
C_8
}
\]

\vspace{1cm}

Thus, in both cases we have two inseparability classes of subgroups, \([C_1]=\{C_1\}\) and \([C_8]=\{C_2,C_4,C_8\}\) and therefore a splitting given by two idempotents \(e_{[1]}=e_1\) and \(e_{[8]}=e_2+e_4+e_8\).

Consider first an \(\cO\)-Mackey functor \(M\) and its splitting.

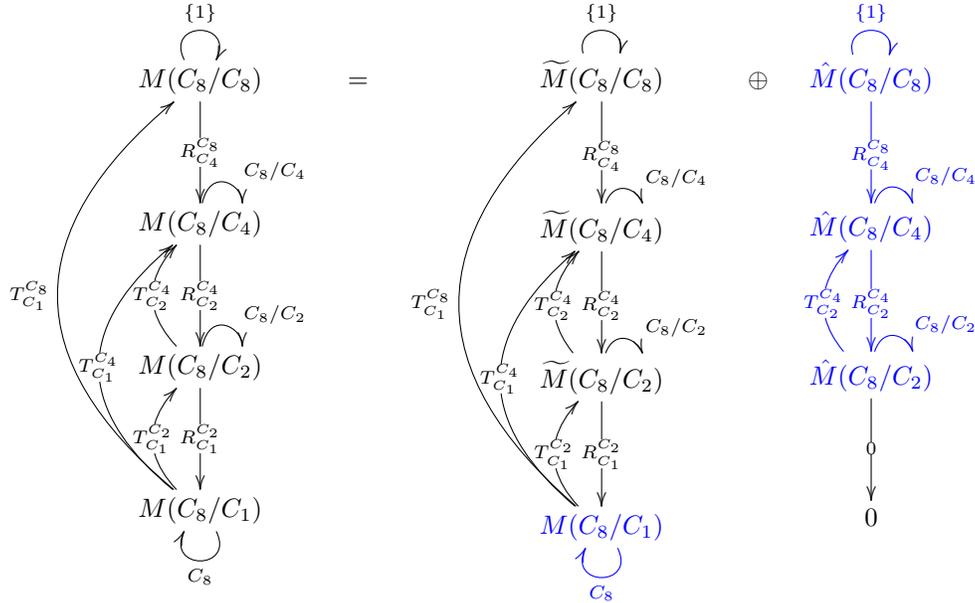
\[
\xymatrix@R+0.5cm@C+0.5cm{
M(C_8/C_8)
\ar@(ul,ur)^{\trivgp}
\ar@/_0.pc/[d]|{R_{C_4}^{C_8}}
\\
M(C_8/C_4)
\ar@(u,r)^{C_8/C_4}
\ar@/_0.pc/[d]|{R_{C_2}^{C_4}}
\\
M(C_8/C_2)
\ar@(u,r)^{C_8/C_2}
\ar@/_0.pc/[d]|{R_{C_1}^{C_2}}
\ar@/^1.5pc/[u]|{T_{C_2}^{C_4}}
\\
M(C_8/C_1)
\ar@(dr,dl)^{C_8}
\ar@/^1.5pc/[u]|{T_{C_1}^{C_2}}
\ar@/^3.2pc/[uu]|{T_{C_1}^{C_4}}
\ar@/^4.5pc/[uuu]^{T_{C_1}^{C_8}}
}
\quad = \quad
\xymatrix@R+0.5cm@C+0.5cm{
\widetilde{M}(C_8/C_8)
\ar@(ul,ur)^{\trivgp}
\ar@/_0.pc/[d]|{R_{C_4}^{C_8}}
\\
\widetilde{M}(C_8/C_4)
\ar@(u,r)^{C_8/C_4}
\ar@/_0.pc/[d]|{R_{C_2}^{C_4}}
\\
\widetilde{M}(C_8/C_2)
\ar@(u,r)^{C_8/C_2}
\ar@/_0.pc/[d]|{R_{C_1}^{C_2}}
\ar@/^1.5pc/[u]|{T_{C_2}^{C_4}}
\\
\blue{M(C_8/C_1)}
\ar@[blue]@(dr,dl)^{\blue{C_8}}
\ar@/^1.5pc/[u]|{T_{C_1}^{C_2}}
\ar@/^3.2pc/[uu]|{T_{C_1}^{C_4}}
\ar@/^4.5pc/[uuu]^{T_{C_1}^{C_8}}
}
\quad \oplus \quad
\xymatrix@R+0.5cm@C+0.5cm{
\blue{\hat{M}(C_8/C_8)}
\ar@[blue]@(ul,ur)^{\blue{\trivgp}}
\ar@[blue]@/_0.pc/[d]|{\blue{R_{C_4}^{C_8}}}
\\
\blue{\hat{M}(C_8/C_4)}
\ar@[blue]@(u,r)^{\blue{C_8/C_4}}
\ar@[blue]@/_0.pc/[d]|{\blue{R_{C_2}^{C_4}}}
\\
\blue{\hat{M}(C_8/C_2)}
\ar@[blue]@(u,r)^{\blue{C_8/C_2}}
\ar@/_0.pc/[d]|{0}
\ar@[blue]@/^1.5pc/[u]|{\blue{T_{C_2}^{C_4}}}
\\
0
}
\]

Now consider an \(\cO^d\)-Mackey functor \(N\) and its splitting.

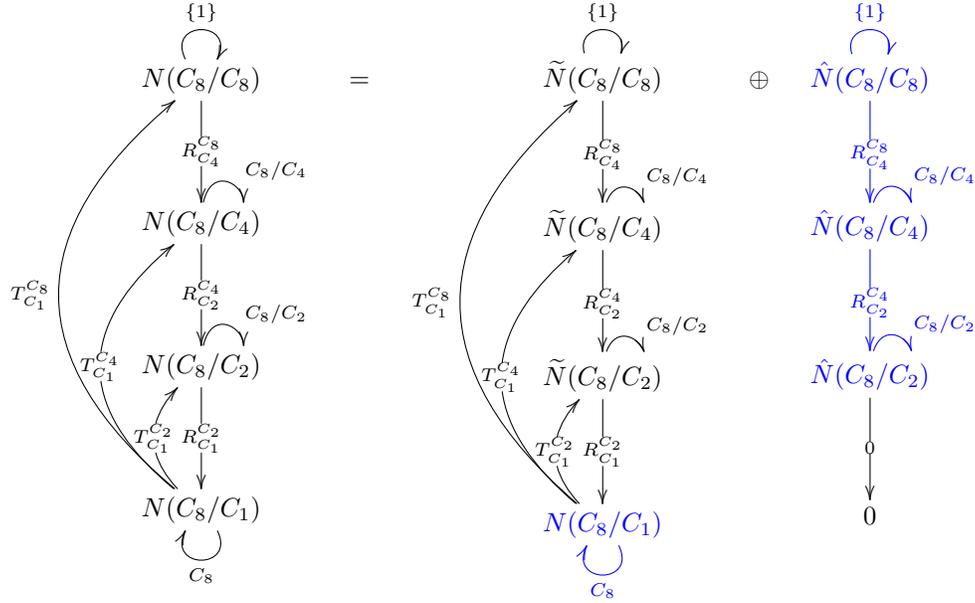
\[
\xymatrix@R+0.5cm@C+0.5cm{
N(C_8/C_8)
\ar@(ul,ur)^{\trivgp}
\ar@/_0.pc/[d]|{R_{C_4}^{C_8}}
\\
N(C_8/C_4)
\ar@(u,r)^{C_8/C_4}
\ar@/_0.pc/[d]|{R_{C_2}^{C_4}}
\\
N(C_8/C_2)
\ar@(u,r)^{C_8/C_2}
\ar@/_0.pc/[d]|{R_{C_1}^{C_2}}
\\
N(C_8/C_1)
\ar@(dr,dl)^{C_8}
\ar@/^1.5pc/[u]|{T_{C_1}^{C_2}}
\ar@/^3.2pc/[uu]|{T_{C_1}^{C_4}}
\ar@/^4.5pc/[uuu]^{T_{C_1}^{C_8}}
}
\quad = \quad
\xymatrix@R+0.5cm@C+0.5cm{
\widetilde{N}(C_8/C_8)
\ar@(ul,ur)^{\trivgp}
\ar@/_0.pc/[d]|{R_{C_4}^{C_8}}
\\
\widetilde{N}(C_8/C_4)
\ar@(u,r)^{C_8/C_4}
\ar@/_0.pc/[d]|{R_{C_2}^{C_4}}
\\
\widetilde{N}(C_8/C_2)
\ar@(u,r)^{C_8/C_2}
\ar@/_0.pc/[d]|{R_{C_1}^{C_2}}
\\
\blue{N(C_8/C_1)}
\ar@[blue]@(dr,dl)^{\blue{C_8}}
\ar@/^1.5pc/[u]|{T_{C_1}^{C_2}}
\ar@/^3.2pc/[uu]|{T_{C_1}^{C_4}}
\ar@/^4.5pc/[uuu]^{T_{C_1}^{C_8}}
}
\quad \oplus \quad
\xymatrix@R+0.5cm@C+0.5cm{
\blue{\hat{N}(C_8/C_8)}
\ar@[blue]@(ul,ur)^{\blue{\trivgp}}
\ar@[blue]@/_0.pc/[d]|{\blue{R_{C_4}^{C_8}}}
\\
\blue{\hat{N}(C_8/C_4)}
\ar@[blue]@(u,r)^{\blue{C_8/C_4}}
\ar@[blue]@/_0.pc/[d]|{\blue{R_{C_2}^{C_4}}}
\\
\blue{\hat{N}(C_8/C_2)}
\ar@[blue]@(u,r)^{\blue{C_8/C_2}}
\ar@/_0.pc/[d]|{0}
\\
0
}
\]

Notice that in both cases the splitting is given by the same idempotents, but the split pieces differ. In the case \(\cO^d\) the blue data for the idempotent piece for \(e_{[8]}=e_2+e_4+e_4\) no longer includes the map \(T_{C_2}^{C_4}\). The blue data for idempotent piece for \(e_1\) is the same in both cases.  
\end{ex}

\doublespacing
\begin{theindex}

  \item \(\cA^{\cO}\), \hyperpage{6}, \hyperpage{18}
  \item \(\cA^{\cO}_{(H)}\), \hyperpage{18, 19}
  \item \(\mA^{\cO}\), \hyperpage{6}, \hyperpage{8}
  \item \(\AboveHMackey[{[H]^{\cO}}]^G\), \hyperpage{18}, 
		\hyperpage{20}, \hyperpage{23--25}, \hyperpage{27}

  \indexspace

  \item \(\Cl^{\cO}(G;\Z)\), \hyperpage{8}, \hyperpage{10}, 
		\hyperpage{12}
  \item \(\Coeff_{\langle H\rangle^{\cO}}^{N}\), \hyperpage{29, 30}
  \item \(\CoIndHO\), \hyperpage{19, 20}, \hyperpage{24}, 
		\hyperpage{26}
  \item \(\cotr_{L_g^H}^{L}(m)\), \hyperpage{24}

  \indexspace

  \item \(e_{[H]^{\cO}}\), \hyperpage{8}, \hyperpage{17}, 
		\hyperpage{28}

  \indexspace

  \item \(\Fin^G\), \hyperpage{5}, \hyperpage{18}
  \item \(\Fin^{G}_{(H)}\), \hyperpage{18}
  \item \(\Fin^{G}_{\cO}\), \hyperpage{6}

  \indexspace

  \item \([H]^\cO\), \hyperpage{11}, \hyperpage{13--15}, \hyperpage{29}
  \item \([H]^\cO\mhyphen\Mackey^G\), \hyperpage{19}, \hyperpage{25}, 
		\hyperpage{29}
  \item \([H]^\cO\mhyphen\Mackey^G_\Q\), \hyperpage{27, 28}
  \item \(\langle H\rangle^{\cO}\mhyphen\Mackey^N\), \hyperpage{30}
  \item \(H_{\cO}(-)\), \hyperpage{10--12}

  \indexspace

  \item \(\IndHO\), \hyperpage{19, 20}, \hyperpage{26}
  \item \(\iota^{\cO}\), \hyperpage{12}

  \indexspace

  \item \(L\geq [H]^{\cO}\), \hyperpage{15--17}, \hyperpage{23}, 
		\hyperpage{25}
  \item \(L_g^{H}\), \hyperpage{16, 17}, \hyperpage{26}

  \indexspace

  \item \(\mu^{\cO}(K,H)\), \hyperpage{9}

  \indexspace

  \item \(\OMackey^G\), \hyperpage{6}
  \item \(\OMackey^G_\Q\), \hyperpage{28}, \hyperpage{30}
  \item \(\OMackey^G_{\mathcal F(H)}\), \hyperpage{18}
  \item \(\cO\), \hyperpage{5}
  \item \(\cO^d\), \hyperpage{14}
  \item \(\Orb^N_{\langle H\rangle^{\cO}}\), \hyperpage{29}
  \item \(\Orb^{G}_{[H]^{\cO}}\), \hyperpage{29}

  \indexspace

  \item \(\restrict{}{[H]^{\cO}}\), \hyperpage{19, 20}, \hyperpage{26}

  \indexspace

  \item \(\Sub(G)\), \hyperpage{5}, \hyperpage{13}
  \item \(\SubOG\), \hyperpage{7}, \hyperpage{11}
  \item \(\Sub_{\langle H\rangle}^{\cO}\), \hyperpage{13}, 
		\hyperpage{16}, \hyperpage{29}

  \indexspace

  \item \(T_{L\geq [H]^{\cO}}\), \hyperpage{16, 17}, \hyperpage{20}, 
		\hyperpage{24}
  \item \(\indtr_{L_g^H}^{L}(m)\), \hyperpage{22}

\end{theindex}

\singlespacing


\end{document}